\documentclass[11pt]{article}

\usepackage[latin1]{inputenc}
\usepackage[T1]{fontenc}
\usepackage[english]{babel}
\usepackage{amsmath,amssymb,amsthm}
\usepackage{mathrsfs} 
\usepackage{graphicx}
\usepackage{bbm}
\usepackage{color}
\usepackage{comment}
\usepackage[textsize=small]{todonotes}       
\usepackage{color}
\usepackage{enumerate}

\newtheorem{theorem}{Theorem}[section]

\newtheorem{definition}[theorem]{Definition}
\newtheorem{lemma}[theorem]{Lemma}
\newtheorem{corollary}[theorem]{Corollary}

\newtheorem{proposition}[theorem]{Proposition}

\newtheorem{remark}[theorem]{Remark}
\newtheorem{cond}[theorem]{Condition}

\newcommand{\be}{\begin{equation}}
\newcommand{\ee}{\end{equation}}
\newcommand{\nn}{\nonumber}

\newcommand{\prob}{\mathbb P}
\newcommand{\expec}{\mathbb E}

\newcommand{\atanh}{{\rm atanh}}

\newcommand{\asinh}{{\rm asinh}}

\newcommand{\dint}{{\rm d}}

\newcommand{\sss}{\scriptscriptstyle}

\numberwithin{equation}{section}

\newcommand{\e}{{\rm e}}
\newcommand{\gs}{\left(G_{\sss N}\right)_{N\geq1}}
\newcommand{\ab}{\alpha(\beta)}
\newcommand{\rem}[1]{}
\newcommand{\bdelta}{\boldsymbol{\delta}}

\newcommand{\wdt}[1] { #1 }

\def\1{{\mathchoice {1\mskip-4mu\mathrm l}      
{1\mskip-4mu\mathrm l}
{1\mskip-4.5mu\mathrm l} {1\mskip-5mu\mathrm l}}}
\newcommand{\indic}[1]{\1_{\{#1\}}}

\newcommand{\q}{Q_{\sss N}} 

\newcommand{\p}{P_{\sss N}} 
\newcommand{\pr}{\mathbb P} 

\newcommand{\GRGw}{\mathrm{GRG}_{\sss N}(\boldsymbol{w})}
\newcommand{\ICW}{\mathrm{CW}_{\sss N}(\boldsymbol{J})}
\newcommand{\cbeta}{\ab}
\newcommand{\cbetaq}{\ab^{2}}
\newcommand{\col}[1]{\textcolor[rgb]{0,0,0}{#1}}
\newcommand{\co}[1]{\textcolor[rgb]{0,0,0}{#1}}
\newcommand{\change}[1]{\textcolor[rgb]{0,0,0}{#1}}

\newcommand{\eqn}[1]{\begin{equation}#1\end{equation}}

\newcommand{\betacan}{\beta_c}

\newcommand{\E}{\mathbb E}

\begin{document}

\title{  Ising critical behavior of inhomogeneous \\Curie-Weiss \change{models} and annealed random graphs }
\author{
Sander Dommers$^{\textup{{\tiny(a)}}}$,
Cristian Giardin\`a$^{\textup{{\tiny(b)}}}$,
Claudio Giberti$^{\textup{{\tiny(c)}}}$,\\
Remco van der Hofstad$^{\textup{{\tiny(d)}}}$,
Maria Luisa Prioriello$^{\textup{{\tiny(b,d)}}}$\;.
\\
{\small $^{\textup{(a)}}$
University of Bologna},
{\small Piazza di Porta San Donato 5, 40126 Bologna, Italy}
\\
{\small $^{\textup{(b)}}$
University of Modena and Reggio Emilia},
{\small via G. Campi 213/b, 41125 Modena, Italy}
\\
{\small $^{\textup{(c)}}$
University of Modena and Reggio Emilia},
{\small Via Amendola 2, 42122 Reggio Emilia, Italy}
\\
{\small $^{\textup{(d)}}$ 
Eindhoven University of Technology,}
{\small P.O. Box 513, 5600 MB Eindhoven, The Netherlands  }
}
\maketitle

\pagenumbering{arabic}

\begin{abstract}
We study the critical behavior \co{for inhomogeneous versions of the Curie-Weiss model}, where the coupling constant  $J_{ij}(\beta)$ for the edge $ij$ on the complete graph is given by $J_{ij}(\beta)=\beta w_iw_j/(
\col{\sum_{k\in[N]}w_k})$. We call the product form of these couplings the rank-1 inhomogeneous Curie-Weiss model. This model also arises (with \col{inverse temperature} $\beta$ replaced by $\sinh(\beta)$) 
\col{from} the annealed Ising model on the generalized random graph. We assume that the vertex weights $(w_i)_{i\in[N]}$ are regular, in the sense that their empirical distribution converges and the second moment converges as well.

We identify the critical 
\col{temperatures} and exponents for these models, as well as a non-classical limit theorem for the total spin at the critical point. These depend sensitively on the number of finite moments of the weight distribution. When 
\col{the} fourth moment of the weight distribution converges, then the critical behavior is the same as on the (homogeneous) Curie-Weiss model, so that the inhomogeneity is weak. When the fourth moment of the weights converges to infinity, and the weights satisfy an asymptotic power law with exponent $\tau$ with $\tau\in(3,5)$, then the critical exponents depend sensitively on $\tau$. 
\col{In addition,} at criticality, the total spin $S_{\sss N}$ satisfies that $S_{\sss N}/N^{(\tau-1)/(\tau-2)}$ converges in 
\col{law}
to some limiting random variable
\col{whose distribution we explicitly characterize}.

\end{abstract}


\section{Introduction}

Universality is a key concept in the theory of phase transitions, with application
to a large variety of physical systems. Informally, universality means that in the 
thermodynamic limit  different systems show common properties close 
to criticality. The theory based on the renormalization group suggests that
systems fall into {\em universality classes}, defined by the values of their {\em critical
exponents} describing the nature of the singularities of measurable
thermodynamic quantities at the critical point.   

In the presence of heterogeneities, e.g.\ {\em spin systems on random graphs} used to model 
interaction on a network \col{\cite{AB,DGM,DGM2,LVVZ}} it is not clear a-priori to 
what extent universality applies.
From the point of view of the structure of the network,
emerging properties 
of real networks have been identified
in several empirical studies in different contexts 
-- social, information, technological and biological networks. 
Many of them are {\em scale free}, with a degree sequence obeying power-law distribution,
and {\em small world}, with short graph-distance among vertices.   
As a consequence power-law random graphs,
i.e., graph sequences where the fraction of nodes that have $k$ neighbors is 
proportional to $k^{-\tau}$ for some $\tau > 1$, 
are often used as mathematical models for real-world networks. 
In this paper we investigate universality for 
spin system on power-law random graphs displaying phase 
transitions.

The issue of universality is  related to the network functionality.
Indeed the occurrence of a thermodynamic phase transition is associated
to a change in macroscopic properties of the networks,
for instance the possibility to reach consensus
in a social network can be related to the occurrence
of a spontaneous magnetization.  
Thus the investigation of different universality classes
for spin systems on random graphs is a relevant 
question with immediate practical relevance
for the network functionality.

Due to the random environment, when considering the Ising model on
the random graphs used to model real networks,  a distinction is required 
between different averaging procedures. Two settings are often studied in the 
literature: the {\em quenched measure}
(graph realizations are studied one-by-one so that they produce 
a random Boltzman-Gibbs measure) and the {\em annealed measure} (all graph 
realizations are considered at once and they give rise to a deterministic 
Boltzman-Gibbs measure). See \col{\cite{GGvdHP, CGZ}} for an extended discussion
of the two settings.

In the paper \cite{DGH2} the {\em quenched critical exponents} have been
rigorously analyzed for a large class of random graph models.
More precisely in \cite{DGH2} it is proved that the critical
exponent $\boldsymbol{\delta}$ (describing the behavior of the magnetization 
at the critical temperature as the external field vanishes), the exponent $\boldsymbol{\beta}$ 
(describing the behavior of the spontaneous
magnetization as the temperature increases to the critical temperature) 
and the exponent $\boldsymbol{\gamma}$ (describing the divergence of the susceptibility as the 
temperature decreases to the critical temperature) take the same values as 
the mean-field Curie-Weiss model whenever the degree distribution
has a finite fourth moment. This includes for instance the case of the  Erd\H{o}s-R\'enyi random graph. 
For power law random graphs, it is proved that for $\tau > 5$ the model is in the mean-field universality class,
whereas the critical exponents are different from the mean-field values for $\tau\in (3,5)$.   

In this paper we provide the 
analysis \col{of the critical behavior} 
but in the annealed setting. Our results are fully
compatible with the universality conjecture. The annealed critical temperature is different
(actually higher) than the quenched critical temperature, but the set of {\em annealed critical exponents}
 that can be rigorously studied are the same as the quenched critical exponents. In the annealed setting our results
are stronger since we are able to show that ${\boldsymbol{\gamma}'} = {\boldsymbol{\gamma}}$.
Here ${\boldsymbol{\gamma}'}$ describes the divergence of the susceptibility as the 
temperature approaches  to the critical temperature from below, and in the quenched
setting we were able to show only that ${\boldsymbol{\gamma}'} \ge {\boldsymbol{\gamma}}$.

A main difference between the quenched and annealed case is that while the analysis of the
quenched measure could be done in great generality, the study of the annealed case is much harder. 
Indeed the results of  \cite{DGH2} are valid for all graph sequences that are locally like a homogeneous random tree 
\col{\cite{BenSch11,DM,DemMon10-bis,DGH,MMS}} and uniformly sparse. For the annealed setting it is not enough
to control the behavior of the model on the typical graph realizations (namely rooted random trees).
For the annealed measure one needs to study exponential functionals of the graphs, 
i.e., questions on large deviations of sparse random graphs that are largely
unsolved. Thus we specialize our analysis of the annealed critical exponents  
to a particular class of random graphs models. This is given by the Generalized
Random Graph models, also called inhomogeneous random graphs of rank-1 in the literature
(see \col{\cite{KBHK,Bia}} for a non-rigorous study).

By exploiting the factorization of the Gibbs measure and the edges independence
we reduce the study of the annealed measure for the Ising model on the Generalized Random Graph 
to the analysis of an inhomogeneous Curie-Weiss model.  
As we shall see, for this model we are able to also study the properties {\em at criticality}. On a sequence of temperatures 
approaching the critical value, we prove the scaling limit for the properly renormalized total \col{spin}.
 As a result,  our 
 \col{findings} extend the analysis of the scaling limit of the standard Curie-Weiss model \cite{Ellis-New01,Ellis-New02, El}
 and provides new asymptotic laws for the (properly renormalized) total spin.

\section{Model definitions and results}


\subsection{Inhomogeneous Curie-Weiss model}

We start by defining the \col{inhomogeneous} Curie-Weiss model.  
This is a generalization of the classical Curie-Weiss model 
in which the strength of the ferromagnetic interaction 
between spins is not spatially uniform. As the standard Curie-Weiss model,
it is defined on the complete graph with vertex set $\left[N\right]:=\left \{1,\ldots,N \right\}$. 
\change{See Table 1 at the end of the paper for a summary of the important notation used 
in this paper.}
\begin{definition}[Inhomogeneous Curie-Weiss model]
Let $\sigma = \{\sigma_i\}_{i\in\left[N\right]}\in \{-1,1\}^{N}$ be  spin variables.
The \col{inhomogeneous} Curie-Weiss model, denoted by $\ICW$, 
is defined by the Boltzmann-Gibbs measure 
\be
\mu_{\sss N}(\sigma) = \frac{\e^{H_{\sss N}(\sigma)}}{Z_{\sss N}} 
\ee
where the Hamiltonian is
\be
H_{\sss N}(\sigma)=\frac{1}{2}\sum_{i,j \in[N]}J_{ij}(\beta)\sigma_{i}\sigma_{j}+B\sum_{i \in[N]}\sigma_{i} 
\ee
and $Z_{\sss N}$ is the normalizing partition function.
Here $\beta$ is the inverse temperature, $B$ is the external magnetic field and
$\mathbf{J} = \{J_{ij}(\beta)\}_{i,j \in[N]}$ are the spin couplings.
\end{definition}
In the above, the interactions $J_{i,j}(\beta)$ 
\col{might} be arbitrary functions
of the inverse temperature (in particular no translation invariance is required), 
provided that the thermodynamic limit is well-defined, i.e., the following limit 
defining the {\em pressure} exists and is finite,
\be
\phi (\beta,B) := \lim_{N\to\infty} \frac{1}{N} \log Z_{\sss N}(\beta,B)\;.
\ee 
In the following we will restrict to the ferromagnetic version of the model, i.e., we will assume
$J_{ij}(\beta) \change{>} 0$. Since the coupling constants $\mathbf{J} = \{J_{ij}(\beta)\}_{i,j \in[N]}$ are \change{positive and possibly} different for different edges, we speak of an {\em inhomogeneous} Curie-Weiss model. We next state our hypotheses on the coupling variables. Each vertex $i \in \left[N\right]$ receives a weight $w_i$, We will take $\mathbf{J} = \{J_{ij}(\beta)\}_{i,j \in[N]}$ such that
	\eqn{
	\label{rank-1-ICW}
	J_{ij}(\beta)=\frac{w_iw_j}{
	\col{\ell_{\sss N}}} \beta,
	\qquad
	\text{where}
	\qquad
	\ell_{\sss N}=\sum_{k\in [n]} w_k.
	}
In the case where $w_i\equiv 1$, our model reduces to the (homogeneous) Curie-Weiss model.
We will call the coupling constants in \eqref{rank-1-ICW} the {\em rank-1 inhomogeneous Curie-Weiss model.} In Section \ref{sec-weights-ass}, we describe the assumptions that we make on the weight sequence $\boldsymbol{w} = (w_i)_{i \in [N]}$.

\medskip

In \cite{GGvdHP} it is shown that  the rank-1 inhomogeneous Curie-Weiss model 
arises in the study of the {\em annealed} Ising model with network of interactions given by the 
rank-1 inhomogeneous random graph, also called the {\em generalized random graph}, which we describe next.

\subsection{Generalized random graph}
In the generalized random graph \col{\cite{vdH,BJR}}, each vertex $i \in \left[N\right]$ receives a weight $w_i$. 
Given the weights, edges are present independently, but the occupation probabilities for different edges are not identical, rather they are 
moderated by the weights of the vertices. We  assume that the weights $\boldsymbol{w} = (w_i)_{i \in [N]}$ are strictly positive  (there is no loss of generality in supposing this, since the vertices with zero weight will be isolated and can be removed from the network).

\begin{definition}[Generalized random graph]
Denote by $I_{ij}$ the Bernoulli indicator that the edge between vertex $i$ and vertex $j$ is present and by 
$p_{ij} = \mathbb{P}\left(I_{ij} = 1\right)$ the edge probability, where different edges are present independently.
Then, the {\em generalized random graph} \col{with vertex set $[N]$},  denoted by $\GRGw$,
is defined by 
	\be\label{eq-prob-edge}
	p_{ij} \, = \, \frac{w_i w_j}{\ell_{\sss N} + w_i w_j},
	\ee
where $\ell_{\sss N}= \sum_{i=1}^N w_i$ is the total weight of all vertices. 
\end{definition}

We have now defined two classes of models that depend on vertex weights $\boldsymbol{w} = (w_i)_{i \in [N]}$. We next state the assumptions on these weights.

\subsection{Assumptions on the vertex weights}
\label{sec-weights-ass}
We study sequences of inhomogeneous Curie-Weiss models and generalized random graphs as  $N\to \infty$. For this, we need to assume that the vertex weight sequences $\boldsymbol{w} = (w_i)_{i \in [N]}$ are sufficiently nicely behaved. 
Let $U_{\sss N} \in [N]$ denote a uniformly chosen vertex in $\GRGw$ and $W_{\sss N} = w_{U_{\sss N}}$ its weight. Then, the following condition defines the asymptotic weight $W$ and set the convergence properties of  $(W_{\sss N})_{N\ge 1}$ to $W$:

\begin{cond}[Weight regularity]  There exists a random variable $W$  such that, as $N\rightarrow\infty$,
\label{cond-WR-GRG}
\begin{enumerate}[(i)]
\item $W_{\sss N} \stackrel{\cal D}{\longrightarrow} W$,
\item $\mathbb{E}[W_{\sss N}^2] =\frac{1}{N}\sum_{i\in [N]} w^{2}_{i}  \rightarrow \mathbb{E}[W^2]< \infty$,
\end{enumerate}
where $ \stackrel{\cal D}{\longrightarrow}$ denotes convergence in distribution. 
Further, we assume that $\mathbb{E}[W]>0$.
\end{cond}
\noindent 
Note that, by uniform integrability, Condition \ref{cond-WR-GRG}(ii) implies that also $\mathbb{E}[W_{\sss N}]=\frac{1}{N}\sum_{i\in [N]} w_{i} \rightarrow \mathbb{E}[W]< \infty$.

Condition \ref{cond-WR-GRG} implies that the sequence $( \GRGw )_{N\ge 1}$ is a 
uniformly sparse tree-like graph with strongly finite mean and with asymptotic degree $D$ 
distributed as a mixed Poisson random variable,
	\be
	\label{degree}
	\prob(D=k)=\mathbb{E}\left[\e^{-W}\frac{W^k}{k!}\right],
	\ee
see e.g., \cite[Chapter 6]{vdH}.

Our results depend sensitively on whether the fourth moment 
\co{of $W$ is finite}. When this is not the case, then we will assume a power-law bound on the tail of the asymptotic weight:

\begin{cond}[Tail of $W$]  The random variable $W$ satisfies either of the following:
\label{cond-WR-GRG-tail}
\begin{enumerate}[(i)]
\item $ \mathbb{E}[W^{4}]<\infty$,
\item $W$ obeys a {\em power law with exponent} $\tau\in (3,5]$, i.e.,  there exist constants $C_{\sss W}>c_{\sss W}>0$ and $w_{0}>1$ such that
\be\label{eq-power-law}
c_{\sss W}w^{-(\tau-1)}\le \pr(W>w)\le C_{\sss W}w^{-(\tau-1)}, \qquad \forall w>w_{0}.
\ee
\end{enumerate}
\end{cond}
\medskip

To prove the results on the scaling limit at criticality 
we will strengthen our assumptions as follows:
\begin{cond}[Tail of $W_{\sss N}$ and deterministic sequences]  
The sequence of weights $(w_i)_{i\in[N]}$ satisfies either of the following:
\label{cond-det}
\begin{enumerate}[(i)]
\item  $\mathbb{E}[W_{\sss N}^4] =\frac{1}{N}\sum_{i\in [N]} w^{4}_{i}  \rightarrow \mathbb{E}[W^4]< \infty$,
\item it coincides with the deterministic sequence
	\be
	\label{eq-precisepowerlaw}
	w_i = \col{c_{w}}\left(\frac{N}{i}\right)^{1/(\tau-1)},
	\ee
for some constant $\col{c_{w}}>0$ and $\tau\in(3,5)$.\end{enumerate}
\end{cond}
We remark that the above deterministic sequence is $N$-dependent (we do not make this dependence explicit) and its limit $W$ satisfies~\eqref{eq-power-law} since $w_i=[1-F]^{-1}(i/N)$, where $F(x)=1-(\col{c_{w}}x)^{-(\tau-1)}$ for $w\geq \col{c_{w}}$. In the next section, we explain what the annealed measure of the Ising model on $\GRGw$ is.

\noindent 


\subsection{Annealed Ising Model}\label{iniz_def}
We first define the annealed Ising model in general on finite graphs with $N$ vertices, then we specialize to $\GRGw$. We denote by  $G_{\sss N}=(V_{\sss N},E_{\sss N})$ a random graph with vertex set $V_{\sss N} = \left[N\right]$ and edge set $E_{\sss N} \subset V_{\sss N} \times V_{\sss N}$.
We denote by  $\q$ the law of the graphs with $N$ vertices.
\begin{definition}[Annealed Ising measure]
\label{measure}
For spin variables  $\sigma = \left (\sigma_1,...,\sigma_{\sss N}\right )$
taking values on the space of spin configurations $\Omega_{\sss N}=\{-1,1\}^N$
\col{the annealed Ising measure is defined by}
\begin{equation}
\label{annealing}
\wdt{\p}(\sigma) = \frac{\q\left(\exp \left[ \beta \sum_{(i,j)\in E_{\sss N}}{\sigma_i \sigma_j} + B \sum_{i \in[N]} {\sigma_i}  \right] \right)}{\q(Z_{G_{\sss N}} \left( \beta, B \right))},
\end{equation}
where
\begin{equation}\label{eq-part-function}
Z_{G_{\sss N}} \left( \beta, B \right)= \sum_{\sigma \in \Omega_{\sss N}} \exp \left[ \beta \sum_{(i,j)\in E_{\sss N}}{\sigma_i \sigma_j} + B \sum_{i \in[N]} {\sigma_i} \right]
\end{equation}
is the partition function. 
\end{definition}
\noindent
With abuse of notation in the following we use the same symbol to denote both a measure and the corresponding expectation. 

\begin{definition}[Annealed thermodynamic quantities]
\noindent For a given $N\in\mathbb{N}$ we introduce the following thermodynamics quantities
at finite volume:
\begin{enumerate}[(i)]



\item  The  {\em annealed pressure}:
\be\label{def-annealed-press}
\wdt{\psi}_{\sss N} (\beta, B)  \,=\, \frac{1}{N} \log \left(\q \left(Z_{\sss N} \left( \beta, B \right)\right)\right).
\ee

\item The  \emph{annealed magnetization}:
\be\label{magn_N}
\wdt{M}_{\sss N} (\beta, B) \,=\, \wdt{\p}\left(\frac{S_{\sss N}}{N}\right),
\ee
where the  \emph{total spin} is defined as
\be
S_{\sss N}= \sum_{i \in [N]} \sigma_i \;.
\ee


\item The \emph{annealed  susceptibility}:
\be\label{susc_N}
\wdt{\chi}_{\sss N}(\beta,B) \,=\, \frac{\partial}{\partial B} \wdt{M}_{\sss N} (\beta, B).
\ee

\end{enumerate}
\end{definition}
\subsection{Annealed Ising Model on GRG}
\label{iniz_def2}
We now specialize the previous definitions to the \col{annealed} Ising Model 
\col{on} the Generalized Random Graph. By assuming  the probability $p_{ij}$ of  each edge in $E_{\sss N}$ is that given  in \eqref{eq-prob-edge}, we 
can compute explicitly the average of the partition function \eqref{eq-part-function}. Indeed,
recalling that $I_{i,j}$ is the indicator of the edge between vertex $i$ and $j$, we can write
\be
\q \left(Z_{\sss N} \left( \beta, B \right)\right) \, 
	 = \, \q\Big( \sum_{\sigma \in \Omega_{\sss N}} 
	\exp \Big[ \beta \sum_{i<j}{I_{ij}\sigma_i \sigma_j} + B \sum_{i \in[N]} {\sigma_i}  \Big] \Big) 
\ee
and, by using the independence of the  variables $I_{i,j}$, we compute \cite{GGvdHP} that
         \be
	\q \left(Z_{\sss N} \left( \beta, B \right)\right)  = \, C\left(\beta \right)\sum_{\sigma \in \Omega_{\sss N}}
	\e^{B \sum_{i \in[N]} {\sigma_i}}\e^{\frac{1}{2}\sum_{i,j \in[N]}{J_{ij}(\beta)\sigma_i \sigma_j}},
	\label{explicit-parti-func-ann}
	\ee
where $C(\beta)>0$  is a constant and the positive \col{couplings} $J_{ij}(\beta)$ are defined as
\be
\label{couplings}
J_{ij}(\beta) = \frac12 \log\left(\frac{\e^{\beta}p_{ij}+(1-p_{ij})}{\e^{-\beta}p_{ij}+(1-p_{ij})}\right)\, . 
\ee
The r.h.s.\ of \eqref{explicit-parti-func-ann} can be seen as  the partition function of an 
\col{inhomogeneous} Curie-Weiss model with couplings $\mathbf{J}$ given by \eqref{couplings}. Thus, the annealed Ising model on the $\GRGw$ is equivalent to such $\ICW$, i.e., the two measures coincide point-wise on the sample space. Our proof \col{(see eq. \eqref{couplings-expansion})} shows that \col{the} $J_{ij}(\beta)$ \col{in \eqref{couplings} are} 
close to the form in \eqref{rank-1-ICW} \col{with $\beta$ replaced by $\sinh(\beta)$}, so that the study of the annealed generalized random graph reduces to the rank-1 ICW model. 
%
%
%
%
%
%
Preliminarily to the statement of our main results we recall the model solution
\col{given in  \cite{GGvdHP}}. 
\change{By symmetry, we always take $B\ge 0$.}
We denote by
$\beta_c$ the {\em annealed critical inverse temperature} defined as
\be
\beta_{c}:=\inf \{\beta > 0: M(\beta, 0^{+})>0 \},
\ee 
where the spontaneous magnetization is given by
\be
M(\beta, 0^{+})=\lim_{B \to 0^{+}} \lim_{N\to\infty} M_{\sss N}(\beta, B)\;.
\ee
%
\begin{theorem}[Thermodynamic limit for annealed Ising on $\GRGw$ and for \col{rank-1} $\ICW$ \cite{GGvdHP}]\label{term_lim_annealed}
Let  $\gs$ be a sequence  of  $\GRGw$ graphs satisfying Condition~\ref{cond-WR-GRG}.
Then,
\begin{enumerate}[(i)]

\item  For all $0\le \beta < \infty$ and for all $B\in \mathbb{R}$, the annealed pressure exists in the thermodynamic limit $N\to\infty$ and is given by
\begin{equation}\label{lim_press_ann}
\wdt{\psi} (\beta, B)  \,:=\,  \lim_{N\rightarrow \infty} \wdt{\psi}_{\sss N} (\beta, B).
\end{equation}

\item   
The {magnetization per vertex} exists in the limit $N\to\infty$
and is given by
\begin{equation}\label{lim_magn_ann}
\wdt{M}(\beta, B) \,:=\, \lim_{N\rightarrow \infty} \wdt{M}_{\sss N} (\beta, B).
\end{equation}
The limit value $\wdt{M}$ equals: $\wdt{M}\change{(\beta,B)} \,=\, \frac{\partial}{\partial B} \wdt{\psi} (\beta, B)$ for 
$B \change{>} 0$, whereas $\wdt{M}=0$ in the region $0< \beta < \beta_c$, $B=0$. More explicitly,
\change{when $B>0$ or $B=0^+$ and $\beta > \beta_c$}
\be\label{eq-magnetization}
\wdt{M}(\beta,B)\, = \, \mathbb{E}\left[\tanh \left(\sqrt{\frac{\sinh\left(\beta\right)}{\mathbb{E}\left[W\right]}}W z^* + B\right) \right],
\ee
where $z^{*}=z^{*}(\beta,B)$ is the  
\change{unique positive solution} of the fixed point equation
\be\label{fixpGRG0}
z \, = \, \mathbb{E}\left[\tanh \left(\sqrt{\frac{\sinh\left(\beta\right)}{\mathbb{E}\left[W\right]}}W z + B\right) \sqrt{\frac{\sinh\left(\beta\right)}{\mathbb{E}\left[W\right]}}\,W\right]\;.
\ee

\item The annealed critical inverse temperature is given by
\be
\label{betac_ann}
\beta_c = \asinh \left(1/\nu\right)\;,
\ee
where 
\be\label{defnu} 
\nu = \frac{\mathbb{E}[W^2]}{\mathbb{E}[W]}\;.
\ee 

\item 
The thermodynamic limit of susceptibility exists and is given by
\begin{equation}\label{lim_susc_ann}
\wdt{\chi}(\beta,B) \,:=\, \lim_{N\rightarrow \infty} \wdt{\chi}_{\sss N} (\beta, B) \,=\, \frac{\partial^2}{\partial B^2} \wdt{\psi} (\beta, B).
\end{equation}

\item For the rank-1 inhomogeneous Curie-Weiss \col{model}  
$\ICW$, 
(i)-(iv) hold with $\beta$ replaced with $\asinh(\beta)$. 
\end{enumerate}
\end{theorem}

Theorem \ref{term_lim_annealed}  shows that 
a phase transition exists for the annealed Ising model on the generalized random graph and the rank-1 inhomogeneous Curie-Weiss model.  For the rank-1 inhomogeneous Curie-Weiss model in the special case where $w_i\equiv 1$, Theorem \ref{term_lim_annealed} reproves the classical result for the Curie-Weiss model. When the weights are inhomogeneous, the critical value is instead given by $\beta_c=1/\nu$.

\change{Let us compare the annealed critical value in \eqref{betac_ann} to that in the quenched setting as derived in \cite{GGvdHP2}.
There, it is proved that the quenched critical value $\beta_c^{\mathrm{qu}}$ equals $\beta_c^{\mathrm{qu}} = \atanh(1/\nu) > \asinh(1/\nu) = \beta_c$.
Thus, the annealed critical value is smaller due to a collaboration of the Ising model and the graph properties.}

In this paper, we analyze the block spin scaling limits at $\beta_{c}$ and we study the universality class of the model.
For this, we define the annealed critical exponents analogous to the random quenched critical exponents as in~\cite{DGH2}:

\begin{definition}[Annealed critical exponents]\label{def-CritExp}
The {\em annealed critical exponents} $\boldsymbol{\beta,\delta,\gamma,\gamma'}$
are defined by:
\begin{align}
\wdt{M}(\beta,0^+) &\asymp (\beta-\betacan)^{\boldsymbol{\beta}},  &&{\rm for\ } \beta \searrow \betacan; \label{eq-def-critexp-beta}\\
\wdt{M}(\betacan, B) &\asymp B^{1/\boldsymbol{\delta}},  &&{\rm for\ } B \searrow 0;
\label{eq-def-critexp-delta}\\
\wdt{\chi}(\beta, 0^+) &\asymp (\betacan-\beta)^{-\boldsymbol{\gamma}},  &&{\rm for\ } \beta \nearrow \betacan; \label{eq-def-critexp-gamma}\\
\wdt{\chi}(\beta, 0^+) &\asymp (\beta-\betacan)^{-\boldsymbol{\gamma'}},  &&{\rm for\ } \beta \searrow \betacan,\label{eq-def-critexp-gamma'}
\end{align}
where we write  $f(x) \asymp g(x)$ if the ratio $f(x)/g(x)$ is bounded away from 0 and infinity for the specified limit.
\end{definition}
\noindent
\col{We remark that, as is customary in the literature, we use the same letter for the inverse temperature $\beta$ and one of the magnetization critical exponent $\boldsymbol{\beta}$. In this paper they are distinguished by the use of the plain, respectively bold, character.}


\subsection{Main results}\label{sec-results}
\noindent 
\col{We start} by proving that the annealed critical exponents for the magnetization and the susceptibility
take the values conjectured in \cite{KBHK}.
\begin{theorem}[Annealed critical exponents] \label{thm-CritExp} Let  $\gs$ be a sequence  of  $\GRGw$ graphs 
\col{fulfilling} Conditions~\ref{cond-WR-GRG}
and \ref{cond-WR-GRG-tail}.
Then, the annealed critical exponents 
\col{defined in Definition~\ref{def-CritExp} using $\beta_c$ given in 	\eqref{betac_ann}} 
exist and satisfy
\begin{center}
{\renewcommand{\arraystretch}{1.2}
\renewcommand{\tabcolsep}{1cm}
\begin{tabular}[c]{c|ccc}
 &  $\tau\in(3,5)$ & $\expec[W^4]<\infty$    \\
\hline
$\boldsymbol{\beta}$  & $1/(\tau-3)$ & $1/2$ \\
$\boldsymbol{\delta}$ & $\tau-2$ & $3$\\
$\boldsymbol{\gamma}=\boldsymbol{\gamma'}$ & $1$ & $1$
\end{tabular}
}
\end{center}
For the boundary case $\tau=5$ there are the following logarithmic corrections for $\boldsymbol{\beta}=1/2$ and $\boldsymbol{\delta}=3$:
	\be
	\label{log-corr-M-tau5}
	\wdt{M}(\beta,0^+) \asymp \Big(\frac{\beta-\betacan}{\log{1/(\beta-\betacan)}}\Big)^{1/2} \quad {\rm for\ } \beta \searrow \betacan,
	\qquad
	\wdt{M}(\betacan, B) \asymp \Big(\frac{B}{\log(1/B)}\Big)^{1/3} \quad {\rm for\ } B \searrow 0.
	\ee
The same results hold for the rank-1 inhomogeneous Curie-Weiss model  $\ICW$, 
\col{the critical exponents being now defined using $\beta_c= 1/\nu$.}
\end{theorem}
\medskip

\begin{remark}[Comparison to the Curie-Weiss model]
For the rank-1 inhomogeneous Curie-Weiss model,  we see that the inhomogeneity does not change the critical behavior when the fourth moment of the weight distribution remains finite, but it does when the fourth moment of the weight distribution increases to infinity. In the latter case, we call the inhomogeneity {\em relevant}.
\end{remark}

\begin{remark}[Comparison to the quenched case]
In \cite{DGH2}, the first two and fourth authors of this paper have shown that the {\em same} critical exponents hold for the {\em quenched} setting of the Ising model on power-law random graphs, such as $\GRGw$, under the assumptions in Conditions~\ref{cond-WR-GRG} and \ref{cond-WR-GRG-tail}. In \cite{DGH2}, however, we only managed to prove a one-sided bound on $\boldsymbol{\gamma}'$. Thus, our results show that for $\GRGw$ both the annealed and quenched Ising model have the same critical exponents, but a different critical value. This is a strong example of universality.
\end{remark}

\begin{remark}[Extension of $\boldsymbol{\gamma}=1$]
The result  $\boldsymbol{\gamma}=1$ holds under more general conditions, i.e., $\expec[W^{2}]<\infty$. See Theorem~\ref{theorem_gamma} below.
\end{remark}
\noindent
From the previous theorem we can also derive the joint scaling of the magnetization as $(\beta,B)\searrow(\betacan,0)$:
\begin{corollary}[Joint scaling in $B$ and $(\beta-\betacan)$]\label{cor-JointScaling}
For $\tau\neq5$,
\be
\wdt{M}(\beta,B) = \Theta\big( (\beta-\betacan)^{\boldsymbol{\beta}} + B^{1/\boldsymbol{\delta}} \big),
\ee
where $f(\beta,B)=\Theta(g(\beta,B))$ means that there exist constants $c_1,C_1>0$ such that $c_1 g(\beta,B) \leq f(\beta,B) \leq C_1 g(\beta,B)$ for all $B\in(0,\varepsilon)$ and $\beta \in (\betacan,\betacan+\varepsilon)$ with $\varepsilon$ small enough.
\noindent
For $\tau=5$,
\be
\wdt{M}(\beta,B) = \Theta\Big(\Big(\frac{\beta-\betacan}{\log{1/(\beta-\betacan)}}\Big)^{1/2}+\Big(\frac{B}{\log(1/B)}\Big)^{1/3}\Big).
\ee
\end{corollary}
\noindent
Our second main result concerns the scaling limit at criticality. 
The next theorem provides the correct scaling and the limit distribution of $S_{\sss N}$ at criticality (for a heuristic derivation of the scaling, see the discussion in Section~\ref{sec-discussion}). For $\GRGw$, we define the inverse temperature sequence 
\be
\label{ctn}
\beta_{c,N} = \asinh (1/\nu_{\sss N}),
\ee
where
\be
\nu_{\sss N} = \frac{\mathbb{E}[W_{\sss N}^2]}{\mathbb{E}[W_{\sss N}]},
\ee
so that $\beta_{c,N} \rightarrow \beta_c$ for $N\rightarrow \infty$. For rank-1 $\ICW$, we replace \co{$\beta$ by $\asinh(\beta)$}, so that $\beta_{c,N} =1/\nu_{\sss N}$. Our main result is the following:

\begin{theorem}[Non-classical limit theorem at criticality]\label{thm-nonclassicalclt}
Let  $\gs$ be a sequence  of  $\GRGw$ graphs satisfying Conditions~\ref{cond-WR-GRG}
and Condition~\ref{cond-det}
and let $\boldsymbol{\delta}$ have the respective values stated in Theorem~\ref{thm-CritExp}. 
Then, there exists a random variable $X$ such that
\be
\frac{S_{\sss N}}{N^{\boldsymbol{\delta}/(\boldsymbol{\delta}+1)}} \stackrel{\cal D}{\longrightarrow} X, \quad \quad {\rm as\ }N\rightarrow\infty,
\ee
where the convergence is w.r.t.\ the  measure $P_{\sss N}$ at inverse temperature $\co{\beta_{c,N} = \asinh (1/\nu_{\sss N})}$ and external field $B=0$.
The random variable $X$ has a density proportional to $\exp(-f(x))$ with
	\be\label{eq-limitingdensity}
	f(x) = \left\{ \begin{array}{ll} \frac{1}{12}\frac{\expec[W^4]}{\expec[W]^4}x^4 
			&\quad {\rm when\ } 		\expec[W^4]<\infty, \\
            	\sum_{i \geq 1} \left(\frac12\left(\frac{\tau-2}{\tau-1} \, x\, i^{-1/(\tau-1)}\right)^2
			-\log\cosh\left(\frac{\tau-2}{\tau-1} \, x\, i^{-1/(\tau-1)}\right)\right) 
					&\quad {\rm when\ } \tau\in(3,5).
						\end{array}\right.
	\ee
The same result holds for the rank-1 inhomogeneous Curie-Weiss model at its critical value $\co{\beta_{c,N} = 1/\nu_{\sss N}}$.
\end{theorem}
\medskip

We will see that in both the case where the fourth moment is finite as well as when it is infinite,
	\be
	\lim_{x\rightarrow\infty}\frac{f(x)}{x^{1+\boldsymbol{\delta}}} = C,
	\ee
with
	\be\label{eq-limitingconstant}
	C = \left\{ \begin{array}{ll} \frac{1}{12} \frac{\expec[W^4]}{\expec[W]^4} 
			&\quad {\rm when\ } \expec[W^4]<\infty, \\
           \left(\frac{\tau-2}{\tau-1}\right)^{\tau-1} \int_0^\infty \left(\frac12 y^{-2/(\tau-1)}
           		- \log\cosh y^{-1/(\tau-1)}\right) \dint y &\quad {\rm when\ } \tau\in(3,5).
						\end{array}\right.
	\ee
This result extends the non-classical limit theorem for the Curie-Weiss model to the annealed $\GRGw$ and the rank-1 $\ICW$.


\subsection{Discussion}\label{sec-discussion}

\paragraph{Random weights.}
Instead of choosing the weights $\boldsymbol{w}$ deterministically, one can also choose the weights i.i.d.\ according to some random variable $W$, with $\expec[W^4]<\infty$. In this case, Condition~\ref{cond-WR-GRG} holds a.s.\ by the laws of large numbers. Hence, if $Q_{\sss N}$ denotes the average over all graphs drawn according to the GRG {\em conditioned} on the weights, then our results also hold a.s. When in the annealing also the average over the weights is taken, then the model becomes unphysical, because the pressure becomes infinite as is proved in~\cite{GGvdHP}.

\paragraph{Critical 
exponents.}
Theorem \ref{thm-CritExp} implies that the annealed exponents are the same as in the quenched case.
Indeed, by \eqref{degree}, the condition $\mathbb{E}(W^4)<\infty$ is equivalent to $\mathbb{E}(K^3)<\infty$, where $K$ is the
forward degree of the branching process describing the local structure of $\GRGw$. 
Thus the  conditions in Theorem \ref{thm-CritExp} defining the universality classes are the same as those in 
Theorem 2.8 in \cite{DGH2}.

\paragraph{Scaling limit of block spin variable.}
In~\cite{GGvdHP}, it is proved that the classical central limit theorem for the total spin $S_{\sss N}$ holds in the 
one-phase region \col{of the annealed Ising model} 
i.e.,
\be
\label{hello}
\frac{S_{\sss N}-P_{\sss N}(S_{\sss N})}{\sqrt{N}} \stackrel{\mathcal{D}}{\longrightarrow} \mathcal{N}(0,\chi), \quad {\rm w.r.t.\ }P_{\sss N}, \quad {\rm as\ }N\rightarrow\infty.
\ee
\change{In \cite{GGvdHP2} we prove the analogous result in the quenched setting.
More precisely, we prove \eqref{hello} for the quenched measure in the quenched uniqueness regime for all
random graphs that are locally tree-like. A prominent example is the $\GRGw$ as studied here.}

\change{At criticality, i.e. for $(\beta,B)=(\beta_c,0)$, the limit in \eqref{hello} is no longer true.} A scaling different from $\sqrt{N}$ has to be used to obtain the scaling limit, and also this limit is not a normal random variable. In~\cite{Ellis-New01,Ellis-New02}, Ellis and Newman prove that for the standard Curie-Weiss model
\be
\frac{S_{\sss N}}{N^{3/4}} \stackrel{\mathcal{D}}{\longrightarrow} X, 
\qquad {\rm as\ }N\rightarrow\infty,
\ee
where $X$ is a random variable with density proportional to $\exp\{-\frac{1}{12}x^4\}$. We extend this result to the
\col{rank-1 inhomogeneous} Curie-Weiss model, \col{and thus to the annealed Ising model}. 
We 
\col{prove} that the scaling with $N^{3/4}$ is also correct when $\expec[W^4]<\infty$, but different for $\tau\in(3,5)$. 
Furthermore we show that when $\expec[W^4]=\infty$, different asymptotic distributions arise in the scaling limit.
We characterize them for the weight deterministic sequence \eqref{eq-precisepowerlaw} in which the weights
follows a precise power-law. Such a sequence is rather generic in the sense that it produces
an asymptotic weight that is also power-law distributed. The analysis
shows that the fluctuations of the total spin decrease as the exponent $\tau$ becomes smaller 
and the distribution seen in the scaling limit has tails proportional to $\e^{-C x^{\tau-1}}$.

\paragraph{Heuristic for the scaling limit.}
To obtain a guess for the correct scaling,  we can use the standard scaling relation between $\boldsymbol{\delta}$ and $\boldsymbol{\eta}$ as in~\cite{El}. On a box in the $d$-dimensional lattice with side lengths $n$, $[n]^d \subset \mathbb{Z}^d$, the exponent $\boldsymbol{\eta}$ satisfies
\be
P^{(d)}_n(S_n^2) \sim n^{d+2-\boldsymbol{\eta}},
\ee
where $P^{(d)}_n$ is the expectation w.r.t.\ the Ising measure on this box and $S_n$ is the sum of all spins inside the box, where it should be noted that there are $n^d$ sites in the box. Hence, to compare this with our setting, we take $N=n^d$ and, with an abuse of notation, let $S_n=S_{\sss N}$. If there is an exponent $\lambda$ such that $S_{\sss N}/N^\lambda$ converges in distribution to a non-trivial limit, then it must also hold that
$
P_{\sss N}\big((S_n/N^\lambda)^2\big)  = P_{\sss N}\big(S_n^2/n^{2d\lambda})  
$
converges. Hence $S_n^2 \sim n^{2d\lambda}$, so that $d+2-\boldsymbol{\eta} = 2d\lambda$. The standard scaling relation $2-\boldsymbol{\eta} = d\frac{\boldsymbol{\delta}-1}{\boldsymbol{\delta}+1}$ 
\cite{El} now suggests that we should choose
\be
\lambda = \frac{\boldsymbol{\delta}}{\boldsymbol{\delta}+1}.
\ee
We prove that this is indeed the correct scaling and we also show that the tail of the density behaves like $\exp\{-Cx^{\boldsymbol{\delta}+1}\}$ as is conjectured on $\mathbb{Z}^d$ (see \cite[{Section~V.8}]{El}). 

\paragraph{Near-critical scaling window.}
\co{Theorem \ref{thm-nonclassicalclt} is proved along the critical sequence $\beta_{c,N}$ 
approaching the critical inverse temperature $\beta_c$ in the limit $N\to\infty$. A different 
scaling limit  might be obtained by working with a sequence near the critical one, 
the so-called {\em near-critical window}, i.e.,  $\beta_{c,N} + \Delta_{\sss N}$
with $\Delta_{\sss N}\to 0$ at an appropriate rate.  As is argued in Section \ref{sec-scaling-window},
it turns out that for the annealed Ising model the width $\Delta_{\sss N}$ 
of the scaling window  is $N^{-(\boldsymbol{\delta}-1)/(\boldsymbol{\delta}+1)}$ and the scaling limit
differs by a quadratic term that appears in in the function $f(x)$ describing the density of $S_{\sss N}/N^{\boldsymbol{\delta}/(\boldsymbol{\delta}+1)}$ in \eqref{eq-limitingdensity}.}

\paragraph{At criticality.}
\co{As a consequence of the previous discussion, we also infer that if one works at
critical inverse temperature $\beta_c$, the scaling limit that will be seen to depend on
the speed at which $\nu_{\sss N}$  approaches $\nu$. Indeed, from \eqref{betac_ann} and 
\eqref{ctn}, one has $\beta_c - \beta_{c,N} = O (\nu - \nu_{\sss N})$. For a natural example given
by the deterministic sequence in Condition \ref{cond-det} (ii) one has that 
when $\tau > 5$ then $\nu - \nu_{\sss N} = o (1/N^{1/2})$  and thus the limiting distribution does not change;
on the contrary when $\tau\in (3,5]$ then  $\nu - \nu_{\sss N}= \zeta N^{-(\boldsymbol{\delta}-1)/(\boldsymbol{\delta}+1)}(1+o(1))$ for some $\zeta\neq 0$,
and thus the distribution changes since we are shifted in the near-critical window. 
See again Section \ref{sec-scaling-window} for more details.
}

\paragraph{Organisation of this paper.}
In Section \ref{sec-pf-thm-CritExp}, we start by deriving the annealed critical exponents in Theorem \ref{thm-CritExp}. In Section \ref{sec-NC-CLT}, we prove our non-classical limit theorems at criticality in Theorem \ref{thm-nonclassicalclt}. We will prove our results only for the annealed $\GRGw$, since the proofs for the rank-1 inhomogeneous Curie-Weiss models are either identical, or simpler. 


\section{Annealed critical exponents: proof of Theorem \ref{thm-CritExp}}
\label{sec-pf-thm-CritExp}
We follow a strategy similar to that in \cite{DGH2}, although the proof in our case is a bit easier since the annealed magnetization is expressed in terms of the deterministic fixed point $z^*$ in~\eqref{fixpGRG0}, whereas in the quenched setting the magnetization is expressed in terms of a fixed point of a distributional recursion.
The proof of Theorem \ref{thm-CritExp} is split into Theorems \ref{theo-beta-delta} dealing with the exponents $\boldsymbol{\beta}$ and  $\boldsymbol{\delta}$ \change{(Section \ref{subsub1})}, Theorem \ref{theorem_gamma} for the exponent $\boldsymbol{\gamma}$ and  Theorem \ref{theorem-gammaprime} for the exponent ${\boldsymbol{\gamma}'}$ \change{(Section \ref{subsub2})}.
Some lemmas and propositions containing preliminary results are \change{also stated and } proved in \change{Section \ref{subsub1}}.

\change{Our analysis of the critical behavior crucially builds on the fixed point equation \eqref{fixpGRG0}.
We apply truncation arguments together with monotonicity (see the proof of Proposition \ref{prop-upperbz} for a prototypical example).
We rely on Taylor expansion properties for the fixed point $z^*$ in \eqref{fixpGRG0} as is customary for the Ising model.
By truncation we mean that we decompose the range on integration of various expectations with respect to the limiting
distribution $W$ according to the size of the fixed point $z^*$ and using asymptotics for truncated moments of $W$.}


\subsection{Magnetization: critical exponents $\boldsymbol{\beta}$ and $\boldsymbol{\delta}$}
\label{subsub1}

We start by showing that the phase transition is continuous.

\begin{lemma}[Continuous phase transition]\label{lem-contphasetrans}
Let $((\beta_\ell,B_\ell))_{\ell\geq1}$ be a sequence with $\beta_\ell$ and $B_\ell$ non-increasing, $\beta_\ell\geq\betacan$ and $B_\ell>0$, and $\beta_\ell\searrow\betacan$ and $B_\ell\searrow 0$ as $\ell\rightarrow\infty$. Then, the solution of \eqref{fixpGRG0} satisfies
\be\label{limz}
\lim_{\ell\rightarrow\infty} z^*(\beta_\ell,B_\ell)=0.
\ee
In particular,
	\be
	\lim_{B\searrow 0} z^*(\betacan,B) = 0, \qquad {\rm and }
	\qquad \lim_{\beta\searrow \betacan} z^*(\beta,0^+)=0.
	\ee
\end{lemma}
\begin{proof}
The existence of the limit \eqref{limz} is a consequence of the monotonicity of $z^{*}(\beta, B)$
and the fact that $z^{*}(\beta,B)\ge 0$ for $B\ge 0$.  Suppose that $\lim_{\ell\rightarrow\infty}z^*(\beta_\ell,B_\ell)=c>0$.  Then, it follows from~\eqref{fixpGRG0} and dominated convergence that
	\begin{align}
	c&=\lim_{\ell\rightarrow\infty} z^*(\beta_\ell,B_\ell) 
	=  \mathbb{E}\left[\tanh \left(\sqrt{\frac{\sinh\left(\betacan\right)}
	{\mathbb{E}\left[W\right]}}W c\right) \sqrt{\frac{\sinh\left(\betacan\right)}
	{\mathbb{E}\left[W\right]}}\,W\right]  < c \sinh(\betacan)\nu =c,
	\end{align}
where we used that $\tanh(x) < x$ for $x>0$ and $\betacan=\asinh(1/\nu)$. This contradiction proves the lemma.
\end{proof}

\noindent
We next show that $z^*$ has the same scaling as we want to prove for $\wdt{M}(\beta,B)$ by proving the upper and lower bounds in Propositions~\ref{prop-upperbz} and~\ref{prop-lowerbz} below. These then allow us to obtain the theorem. But first we state some properties for truncated moments of $W$ in the following lemma:

\begin{lemma}[Truncated moments of $W$]\label{lem-truncatedmoments}
Assume that $W$ obeys a power law for some $\tau>1$,  see item (ii) in  Condition \ref{cond-WR-GRG-tail}. Then
there exist constants $c_{a,\tau},C_{a,\tau} >0$ such that,
as $\ell \rightarrow \infty,$
	\be\label{eq-trunc-moments-w}
	\begin{cases}
	c_{a,\tau}\ell^{a-(\tau-1)}  \\
	c_{\tau-1,\tau}\log \ell 
	\end{cases} \leq
	\expec\left[W^a \indic{W \leq \ell}\right] \leq
	\begin{cases}
	C_{a,\tau}\ell^{a-(\tau-1)} &\text{when } a > \tau-1,\\
	C_{\tau-1,\tau}\log \ell &\text{when }  a=\tau-1.
	\end{cases}
	\ee
and, when $a<\tau-1$,
	\be\label{eq-trunc-moments-w-2}
	\expec\left[W^a \indic{W > \ell}\right] \leq C_{a,\tau}\ell^{a-(\tau-1)}.
	\ee
\end{lemma}
\begin{proof} 
The proof is similar to that of \cite[Lemma 3.4]{DGH2}.
\end{proof}

\noindent
In the following we write $c_{i}$, $C_{i}$, $i\ge 1$ for constants that only depend on $\beta$ and on moments of  $W$ and satisfy 
	\be
	0<\liminf_{\beta \searrow \beta_{c}} c_{i}(\beta) \le \limsup_{\beta \searrow \beta_{c}}  c_{i}(\beta)<	\infty , 
	\ee
and the same holds for $C_{i}$. The constants $C_{i}$ appear in upper bounds and $c_{i}$ in lower bounds.  Furthermore, we write
$e_{i}$, $i\ge 1$ for error functions that depend on $\beta$, $B$ and on moments of $W$, and satisfy
	\be\label{constants}
	\limsup_{B\searrow 0} e_{i}(\beta, B) < \infty\quad 
	\mbox{and}\quad 
	\lim_{B\searrow 0} e_{i}(\beta_{c}, B)=0.
	\ee 
Here, the subscript $i$  is just a label for constants and error functions.

Further, we introduce the following notation that will be used extensively in the following:
	\be
	\alpha(\beta):=\sqrt{\frac{\sinh(\beta)}{\expec[W]}}.
	\ee

\begin{proposition}[Upper bound on $z^*$]\label{prop-upperbz}
Let $\beta\geq\betacan$ and $B>0$. Then, there exists a $C_1>0$ such that
	\be\label{eq-upbdzE4}
	z^* \leq \sqrt{\expec[W]\sinh(\beta)}B+\sinh(\beta)\nu z^*-C_1 z^{*\boldsymbol{\delta}},
	\ee
where $\boldsymbol{\delta}$ takes the values as stated in Theorem~\ref{thm-CritExp}.
For $\tau=5$,
	\be\label{eq-UpperExi-tau5}
	z^* \leq \sqrt{\expec[W]\sinh(\beta)}B+\sinh(\beta)\nu z^*-C_1 z^{*3}\log\left(1/z^*\right).
	\ee
\end{proposition}

\begin{proof}
We frequently use that $\tanh(B)\leq B$. A Taylor expansion around $x=0$ gives that, for some $\zeta\in(0,x)$,
	\begin{align}\label{eq-uptanh3rd}
	\tanh(x+B)&= \tanh(B) +(1-\tanh^2(B))x-\tanh(B)(1-\tanh^2(B))x^2\nn\\
	&\qquad -\frac13 (1-\tanh^2(\zeta+B))x^3 + \tanh(\zeta+B)(1-\tanh^2(\zeta+B))x^3 \nn\\
	& \leq B + x - \frac13 x^3 + \frac43 \tanh(x+B)x^3,
	\end{align}
where we also used that $\tanh(x)\leq 1$. If we now assume that $x+B \leq \atanh\frac18$, then
	\be\label{eq-uptanh3rd2}
	\tanh(x+B) \leq B + x -\frac16 x^3.
	\ee
We apply this result to~\eqref{fixpGRG0} where $x=\ab W z^*$, which we force to be at most $\atanh\frac18$ by introducing an indicator function as follows:
	\begin{align}
	z^* &\leq \expec\left[\left(B+\ab W z^*\right) \ab \,W\right]\nn\\
	&\qquad +\mathbb{E}\left[\left\{\tanh \left(\ab W z^* + B\right)-\left(B+\ab  W z^*\right)\right\} 
	\ab \,W \indic{\ab W z^*+B\leq \atanh\frac18}\right],
	\end{align}
since $\tanh(B+x)\leq B+x$. Hence, using~\eqref{eq-uptanh3rd2},
	\be\label{upper_b_z}
	z^* \leq \sqrt{\expec[W]\sinh(\beta)}B+\sinh(\beta)\nu z^*
	-\frac16 \ab^{4} \expec\left[W^4 \indic{\ab W z^* +B \leq \atanh\frac18}\right] z^{*3}.
	\ee
For $\expec[W^4]<\infty$, this is indeed of the form~\eqref{eq-upbdzE4} and we are done. If $\tau\in(3,5)$, then it follows from Lemma~\ref{lem-truncatedmoments} that
	\be
	\expec\left[W^4 \indic{\ab W z^*+B\leq \atanh\frac18}\right] 
	\geq c_{4,\tau} \left( \frac{\ab }{(\atanh\frac18-B)}z^*\right)^{\tau-5},
	\ee
which proves the proposition for $\tau\in(3,5)$. The proof for $\tau=5$ is similar and we omit it.
\end{proof}

We now  proceed with the lower bound:

\begin{proposition}[Lower bound on $z^*$]\label{prop-lowerbz}
Let $\beta\geq\betacan$ and $B>0$. Then, there exists a $c_1>0$ such that
	\be\label{lower_b_z}
	z^* \geq \sqrt{\expec[W]\sinh(\beta)}B+\sinh(\beta)\nu\,z^*-c_1 z^{*\boldsymbol{\delta}}-Be_1,
	\ee
where $\boldsymbol{\delta}$ takes the values as stated in Theorem~\ref{thm-CritExp}.
For $\tau=5$,
	\be\label{eq-LowerExi-tau5}
	z^* \geq \sqrt{\expec[W]\sinh(\beta)}B+\sinh(\beta)\nu z^*-c_1 z^{*3}\log\left(1/z^*\right)-Be_1.
	\ee
\end{proposition}

\begin{proof}
As in~\eqref{eq-uptanh3rd} we can bound
	\be\label{eq-lowtanh3rd}
	\tanh(x+B) \geq B+x-\frac13 x^3 -B(B+B x+x^2),
	\ee
where we have used that $B-B^2\leq \tanh(B) \leq B$.
For $\expec[W^4]<\infty$, we can immediately use this to obtain
	\be 
	z^* \geq \sqrt{\expec[W]\sinh(\beta)}B+\sinh(\beta)\nu z^* -c_1z^{*3}-B e_1,
	\ee
where
	\be\label{popi}
	c_1 = \frac13\sinh^2(\beta)\frac{\expec[W^4]}{\expec[W]^2},
	\ee
and
	\be
	e_1 = B \sqrt{\mathbb{E}\left[W\right]\sinh\left(\beta\right)} 
	+ B \sinh(\beta)\nu z^* 
	+ \left(\frac{\sinh\left(\beta\right)}{\mathbb{E}\left[W\right]}\right)^{3/2}\expec[W^3]z^{*2}.
	\ee
All terms in $e_1$ indeed converge to $0$ in the appropriate limit, because of Lemma~\ref{lem-contphasetrans}. 

For $\tau\in(3,5)$, we rewrite $z^*$ as
	\begin{align}
	z^* &= \sqrt{\expec[W]\sinh(\beta)}B+\sinh(\beta)\nu z^* 
		\nn\\
	&\qquad+\mathbb{E}\left[\left\{\tanh \left(\ab W z^* + B\right)-\left(B+\ab  W z^*\right)\right\} 
	\ab \,W \left(\indic{W\leq 1/z^*}+\indic{W> 1/z^*} \right)\right].
	\end{align}
The case where $W\leq1/z^*$ can be treated as above. This gives
	\begin{align}\label{eq-lowzsmallW}
	\mathbb{E}\bigg[\bigg\{&\tanh \left(\ab W z^* + B\right)
	-\left(B+\ab  W z^*\right)\bigg\} \ab \,W \indic{W\leq 1/z^*}\bigg] \nn\\
	&\geq - \frac13\sinh^2(\beta)\frac{\expec[W^4 \indic{W\leq 1/z^*}]}{\expec[W]^2} z^{*3} - Be_2,
	\end{align}
where
	\be
	e_2 = B \ab \expec[W\indic{W\leq 1/z^*}] 
	+ B \sinh(\beta)\frac{\expec[W^2\indic{W\leq 1/z^*}]}{\expec[W]}z^* 
	+ \ab^{3}\,\expec[W^3\indic{W\leq 1/z^*}]z^{*2}.
	\ee
By Lemma~\ref{lem-truncatedmoments},
	\be
	\expec[W^4 \indic{W\leq 1/z^*}] \leq C_{4,\tau} z^{*\tau-5},
	\ee
so that indeed~\eqref{eq-lowzsmallW} is bounded from below by
	\be
	-\col{c_2} z^{*\tau-2} -Be_2.
	\ee
\col{with 
\be
c_2 = \frac13\sinh^2(\beta)\frac{C_{4,\tau}}{\expec[W]^2}.
\ee	
}
Using Lemma's~\ref{lem-contphasetrans} and~\ref{lem-truncatedmoments}, one can also show that all terms in $e_2$ indeed converge to $0$ in the appropriate limit.

It remains to bound the term where $W>1/z^*$. For this we use that $\tanh(x+B)\geq 0$:
	\begin{align}\label{eq-lowzlargeW}
	\mathbb{E}\bigg[\bigg\{&\tanh \left(\ab W z^* + B\right)
	-\left(B+\ab  W z^*\right)\bigg\} \ab \,W \indic{W> 1/z^*}\bigg] \nn\\
	&\geq - \sinh(\beta)\frac{\expec[W^2 \indic{W> 1/z^*}]}{\expec[W]} z^{*} - Be_3,
	\end{align}
where
	\be
	e_3 = \ab  \expec[W \indic{W> 1/z^*}].
	\ee
By Lemma~\ref{lem-truncatedmoments},
	\be
	\expec[W^2 \indic{W> 1/z^*}] \leq C_{2,\tau} z^{*\tau-3},
	\ee
again giving the right scaling. \col{As a consequence \eqref{eq-lowzlargeW} is bounded from below by $-{c_3} z^{*\tau-2} -Be_3$ with
\be
c_3 = \frac13\sinh(\beta)\frac{C_{2,\tau}}{\expec[W]^2}.
\ee
} Similarly,
	\be
	e_3 \leq \ab C_{1,\tau} z^{*\tau-2},
	\ee
which indeed converges to 0. We conclude that \eqref{lower_b_z} holds with $c_1=c_2+c_3$ \col{and $e_1=e_2+e_3$}.
\end{proof}
\medskip

The upper and lower bounds on $z^*$ in the previous two propositions allow us to prove that the critical exponents take the values stated in Theorem~\ref{thm-CritExp}:

\begin{theorem}[Values of  $\boldsymbol{\beta}$ and  $\boldsymbol{\delta}$]\label{theo-beta-delta}
The critical exponents  $\boldsymbol{\beta}$ and  $\boldsymbol{\delta}$ equal the values as stated in Theorem \ref{thm-CritExp} when $\expec[W^{2}]<\infty$ and $\tau \in (3,5)$. Furthermore, for $\tau=5$ \eqref{log-corr-M-tau5} holds.
\end{theorem}

\begin{proof}
{\bf Proof for exponent}  $\boldsymbol{\beta}$.
We start by giving upper bounds on the magnetization. From~\eqref{eq-magnetization} it follows that
	\be\label{up_b_M}
	\wdt{M}(\beta,B)\, = \, \mathbb{E}\left[\tanh \left(\ab W z^* + B\right) \right]
	 \leq B+ \sqrt{\mathbb{E}\left[W\right]\sinh\left(\beta\right)} z^*.
	\ee
We first analyze $\boldsymbol{\beta}$ and hence take the limit $B\searrow0$ for $\beta>\betacan$. This gives 
	\be\label{emme_zero_piu}
	\wdt{M}(\beta,0^+) \leq \sqrt{\mathbb{E}\left[W\right]\sinh\left(\beta\right)} z^*_0,
	\ee
where we write $z^*_0 = \lim_{B\searrow0}z^*$. Since $\wdt{M}(\beta,0^+)>0$ by the definition of $\betacan$, the same must be true for $z^*_0$. We will deal first with the cases 
$\tau \in (3,5)$ and $\mathbb{E}[W^{4}]<\infty$. \col{Taking}  the limit $B\searrow0$ in~\eqref{eq-upbdzE4} and dividing by $z^*_0$, we get for $\tau\neq5$
	\be
	C_1 z^{*\boldsymbol{\delta}-1}_0 \leq \sinh(\beta) \nu  -1,
	\ee
and hence, observing that $\boldsymbol{\beta}=1/(\boldsymbol{\delta-1})$,
	\be\label{upperz}
	z^{*}_0 \leq C_1^{-\boldsymbol{\beta}} \left(\sinh(\beta) \nu  -1\right)^{\boldsymbol{\beta}}.
	\ee
From a Taylor expansion of $\sinh(\beta)$ around $\betacan=\asinh(\expec[W]/\expec[W^2])$ it follows that
	\be\label{zorro}
	\sinh(\beta) \nu  -1 \leq \cosh(\beta)\nu (\beta-\betacan).
	\ee
Hence,
	\be
	\wdt{M}(\beta,0^{+})\leq 
	\sqrt{\mathbb{E}\left[W\right]\sinh\left(\beta\right)}
	C_1^{-\boldsymbol{\beta}}\left(\cosh(\beta)\nu \right)^{\boldsymbol{\beta}}
	(\beta-\betacan)^{\boldsymbol{\beta}},
	\ee
so that it is easy to see that
	\be
	\limsup_{\beta\searrow\betacan}\frac{\wdt{M}(\beta,0^{+})}
	{(\beta-\betacan)^{\boldsymbol{\beta}}} <\infty.
	\ee
The lower bound can be obtained in a similar fashion. Starting from $\tanh x \ge x - x^2$ and taking the limit $B\searrow0$ for $\beta>\betacan$  in \eqref{eq-magnetization}, we obtain
	\be\label{mb0lower}
	\wdt{M}(\beta,0^+) \geq \sqrt{\mathbb{E}\left[W\right]\sinh\left(\beta\right)} z^*_0
	-\sinh(\beta)\nu  {z^*_0}^2 .
	\ee
Again, starting from the lower bound \eqref{lower_b_z}, taking $B\searrow0$ and dividing by $z^*_0$
	\be\label{lowerz}
	z^{*}_0 \geq c_1^{-\boldsymbol{\beta}} \left(\sinh(\beta) \nu  -1\right)^{\boldsymbol{\beta}},
	\ee
and, by a Taylor expansion around $\beta_c$,
	\be\label{taylorexpx}
	\sinh(\beta)\nu -1= \cosh(\beta_c)\nu (\beta-\beta_c)+O((\beta-\beta_c)^2).
	\ee
Using \eqref{upperz}, \eqref{lowerz} and \eqref{taylorexpx} in \eqref{mb0lower} we obtain:
	\begin{align}
	\label{low_bound_emme0}
	\wdt{M}(\beta,0^+) 
	&\ge \sqrt{\mathbb{E}\left[W\right]\sinh\left(\beta\right)}  
	c_1^{-\boldsymbol{\beta}} 
	\left [ \cosh(\beta_c)\nu (\beta-\beta_c)+O((\beta-\beta_c)^2) \right ]^{\boldsymbol{\beta}} \nn \\
	&-\sinh(\beta)\nu C_1^{-2\boldsymbol{\beta}} \left [ \cosh(\beta_c)\nu (\beta-\beta_c)
	+O((\beta-\beta_c)^2) \right ]^{2\boldsymbol{\beta}},
	\end{align}
which shows that also
	\be
	0< \liminf_{\beta\searrow\betacan}\frac{\wdt{M}(\beta,0^{+})}
	{(\beta-\betacan)^{\boldsymbol{\beta}}},
	\ee
concluding the proof for the exponent $\boldsymbol{\beta}$ in the cases $\tau \in (3,5) $ and $\expec[W^4]<\infty$ . In the case $\tau=5$ we can prove the upper bound for $M(\beta, 0^{+})$ in a similar fashion, i.e., taking the limit $B\searrow0$ for $\beta>\betacan$ of  \eqref{eq-UpperExi-tau5} and dividing by $z^{*}_{0}$. This yields  to
	\be\label{zorro2}
	{z^{*}_{0}}^{2}\le \frac{\sinh(\beta) \nu  -1}{C_{1}\log (1/z^{*}_{0})}
	\le  \frac{ \cosh(\beta)\nu (\beta-\betacan)}{C_{1}\log (1/z^{*}_{0})}
	\le \widetilde{C} \frac{(\beta-\betacan)}{\log (1/z^{*}_{0})},
	\ee
where \eqref{zorro} has been used in order to obtain the second inequality and $\cosh(\beta)$ has been bounded in a right neighborhood of $\betacan$ to obtain the third inequality. Since $x \mapsto 1/\log(1/x)$ is increasing in (0,1) and $z^{*}_{0}\le C(\beta-\betacan)^{\frac 1 2}$ for some $C>0$,\footnote{The proof of $z^{*}_{0}\le C(\beta-\betacan)^{\frac 1 2}$ can be obtained  by rewriting  \eqref{zorro2} as $-{z^{*}_{0}}^{2}\log {z^{*}_{0}}^{2} \le k (\beta - \betacan)$, for some $k>0$. Since  $w < -w \log w$ for  $w< 1/\e$, we conclude that for $\beta - \betacan>0$ small enough, the previous inequality  gives ${z^{*}_{0}}^{2}< k(\beta-\betacan)$.} form \eqref{zorro2} we obtain:
 	\be
	 {z^{*}_{0}}^{2}\le \frac{\widetilde{C}(\beta-\betacan)}{C_{1}\log(1/[C(\beta-\betacan)^{1/2}]) }.
	 \ee
 The previous inequality together with \eqref{emme_zero_piu}, proves the upper bound
	 \be
	 \limsup_{\beta\searrow\betacan} \frac{\wdt{M}(\beta,0^{+})}
	 { \left (\frac{\displaystyle \beta-\betacan}{\displaystyle \log(1/(\beta-\betacan)) }\right)^{1/2}}< 		\infty.
 	\ee
The lower bound can be obtained in the same way. Indeed, from \eqref{eq-LowerExi-tau5} in the limit $B\rightarrow 0$, we obtain, for some positive constants $\widetilde C$ and $\widehat C$
	\be\label{zorro3}
	{z^{*}_{0}}^{2}\ge \frac{\sinh(\beta) \nu  -1}{C_{1}\log (1/z^{*}_{0})}
	\ge \widetilde{C} \frac{(\beta-\betacan)}{\log (1/z^{*}_{0})} 
	\ge {\widehat C} \frac{(\beta-\betacan)}{ \log(1/(\beta-\betacan))}, 
	\ee
where, once more, we have used that $x \mapsto 1/\log(1/x)$ is increasing in (0,1) and  the bound $z^{*}_{0}\ge C(\beta-\betacan)^{1/(2-\varepsilon)}$ for some $C>0$ and any $0<\varepsilon < 2$.\footnote{The proof of  the inequality $z^{*}_{0}\ge C(\beta-\betacan)^{1/(2-\varepsilon)}$, for $0<\varepsilon < 2$  can be obtained starting from the rightmost inequality of \eqref{zorro3} combined with the fact that $\log 1/x \le A_{\varepsilon} x^{-\varepsilon}$ for all $x\in (0,1)$ and any $\varepsilon >0$.}
 The previous inequality plugged in \eqref{mb0lower} gives
	 \be
	 \liminf_{\beta\searrow\betacan} \frac{\wdt{M}(\beta,0^{+})}
	 { \left (\frac{\displaystyle \beta-\betacan}{\displaystyle \log(1/(\beta-\betacan)) }\right)^{1/2}} > 0,
 	\ee
concluding the proof for $\tau=5$.
\qed

\vspace{0.3cm}
\noindent
{\bf Proof for exponent}  $\boldsymbol{\delta}$.
We continue with the analysis for $\boldsymbol{\delta}$. Setting $\beta=\beta_c$ in \eqref{eq-upbdzE4}, we obtain
	\be\label{up_b_zcrit}
	z^*(\beta_c,B)\le( C_1 \sqrt{\expec[W]})^{-\frac 1 {\boldsymbol{\delta}}} 
	B^{ 1/ {\boldsymbol{\delta}}}.
	\ee
Using this inequality in \eqref{up_b_M} with $\beta=\beta_c$, we obtain
	\be
	M(\beta_c,B) \le B+ \frac{\expec[W] C_1^{-\frac 1 {\boldsymbol{\delta}}} }
	{(\sqrt{\expec[W^2]})^{1+\frac 1 {\boldsymbol{\delta}}}}B^{1 /{\boldsymbol{\delta}}},
	\ee
which proves that 
	\be
	\limsup_{B\searrow 0} \frac{M(\beta_c,B) }{B^{ 1/ {\boldsymbol{\delta}}}} < \infty
	\ee
since $ {\boldsymbol{\delta}}>1$. Inequality \eqref{lower_b_z} with $\beta=\beta_c$ gives
	\be\label{zeta_cri_delta}
	z^*(\beta_c,B)\ge c_1^{-1 / {\boldsymbol{\delta}} } 
	\left (\frac{1}{\sqrt{\expec[W^2]}} -e_1(\beta_c,B)\right )^{1/{\boldsymbol{\delta}}} 
	B^{1/ \boldsymbol{\delta}}\;.
	\ee
This estimate, along with \eqref{up_b_zcrit}, will be used in the lower bound  of the magnetization at $\beta=\beta_c$  obtained by $\tanh x \ge x - x^2$:
	\be
	M(\beta_c,B) \ge B 
	+ \frac{{\expec [W]}}{\sqrt{\expec[W^{2}]}}(1-2B) z^*(\beta_c,B)-  {z^*(\beta_c,B)}^{2} -B^2,
	\ee
giving, for $B>0$ small,
	\begin{align}\label{lower_bound_emme_cri}
	M(\beta_c,B) & \ge B  + \frac{{\expec [W]}}{\sqrt{\expec[W^{2}]}}(1-2B) 
	c_1^{-1 / {\boldsymbol{\delta}} } \left (\frac{1}{\sqrt{\expec[W^2]}} 
	-e_1(\beta_c,B)\right )^{1/{\boldsymbol{\delta}}} B^{1/ \boldsymbol{\delta}}\nn \\
	& - ( C_1 \sqrt{\expec[W]})^{- 2 /{\boldsymbol{\delta}}} B^{ 2/ {\boldsymbol{\delta}}} -B^{2}.
	\end{align}
Recalling that $\lim_{B\searrow 0} e_1(\beta_c,B)=0$ and ${\boldsymbol{\delta}} >1$, the previous bound gives
	\be
	\liminf_{B\searrow 0} \frac{M(\beta_c,B) }{B^{ 1/ {\boldsymbol{\delta}}}} 
	\ge  \sqrt{\expec [W]} (c_1 \sqrt{\expec[W^2]}  )^{-1/\delta} > 0,
	\ee
which concludes the proof for ${\boldsymbol{\delta}} $ in the cases $\tau \in (3,5) $ and $\expec[W^2]<\infty$. The analysis for $\tau=5$ can be performed in a similar way as for $ {\boldsymbol{\beta}}$. 
\end{proof}

\noindent
\begin{proof}[Proof of Corollary~\ref{cor-JointScaling}] The proof can be simply adapted as in~\cite[Corollary 2.9]{DGH2}.
\end{proof}


\subsection{Susceptibility: critical exponents $\boldsymbol{\gamma}$ and $\boldsymbol{\gamma'}$}
\label{subsub2}
We now analyze the susceptibility and compute the critical exponents $\boldsymbol{\gamma}$ and $\boldsymbol{\gamma'}$. We start by computing the former under more general conditions than those of Theorem \ref{thm-CritExp}.
\begin{theorem}[Value of  $\boldsymbol{\gamma}$]\label{theorem_gamma}
For $\expec[W^2]<\infty$, 
	\be\label{lim_chi_gamma}
	\lim_{\beta \nearrow \betacan}\wdt{\chi}(\beta,0^+) (\betacan-\beta) 
	\,=\, \frac{\expec[W]^2}{\expec[W^2]}	\tanh(\betacan),
	\ee
so that $\boldsymbol{\gamma}=1$.
\end{theorem}

\begin{proof}
From Theorem~\ref{term_lim_annealed} it follows that in the one-phase region, i.e., for $\beta<\beta_c$ or $B\neq 0$,
	\be\label{eq-chiexplicit}
	\wdt{\chi}(\beta,B) = \frac{\partial}{\partial B} \wdt{M}(\beta,B) 
	= \mathbb{E}\left[\left(1+\ab W\frac{\partial z^*}{\partial B}\right)
	\left( 1-\tanh^2 \left(\ab W z^* + B\right)\right) \right].
	\ee
We can also compute the derivative of $z^*$ by taking the derivative of~\eqref{fixpGRG0}:
	\be\label{eq-derivative-z-B}
	\frac{\partial z^*}{\partial B} 
	= \mathbb{E}\left[\left(\ab W+\ab^{2}{W^2}\frac{\partial z^*}{\partial B}\right)
	\left(1-\tanh^2 \left(\ab W z^* + B\right)\right) \right].
	\ee
If we take the limit $B\searrow0$ for $\beta<\betacan$, then the $\tanh^2(\cdot)$ term vanishes, since by definition of $\betacan$ it holds that $z^*_0\equiv\lim_{B\searrow0} z^* = 0$. Hence, if we write 
	\be
	\frac{\partial z^*_0}{\partial B} = \lim_{B\searrow 0} \frac{\partial}{\partial B} z^*(\beta,B),
	\ee 
then~\eqref{eq-derivative-z-B} simplifies to
	\be
	\frac{\partial z^*_0}{\partial B} = \sqrt{\expec[W]\sinh(\beta)}+\sinh(\beta)\nu \frac{\partial z^*_0}{\partial B}.
	\ee
Solving for $\frac{\partial z^*_0}{\partial B}$ gives
	\be
	\frac{\partial z^*_0}{\partial B} = \frac{\sqrt{\expec[W]\sinh(\beta)}}{1-\sinh(\beta)\nu }.
	\ee
Also taking the limit $B\searrow0$ in~\eqref{eq-chiexplicit} and using the above gives
	\be
	\wdt{\chi}(\beta,0^+) = 1 + \frac{\expec[W]\sinh(\beta)}{1-\sinh(\beta)\nu }.
	\ee
From a Taylor expansion around $\beta_c$, we get that
	\be\label{eq-Taylor-sinh}
	\sinh(\betacan) - \cosh(\betacan)(\betacan-\beta) \leq \sinh(\beta) 
	\leq \sinh(\betacan) - \cosh(\beta)(\betacan-\beta),
	\ee
so that
	\be
	1 + \frac{\expec[W]^2\sinh(\beta)}{\expec[W^2]\cosh(\betacan)(\betacan-\beta)} 
	\leq \wdt{\chi}(\beta,0^+) \leq 1 + \frac{\expec[W]^2\sinh(\beta)}{\expec[W^2]\cosh(\beta)(\betacan-\beta)},
	\ee
since $\sinh(\betacan) \nu  =1$.
Hence, (\ref{lim_chi_gamma}) follows.
\end{proof}

\noindent
We now analyze $\boldsymbol{\gamma'}$:
\begin{theorem}[Value of  ${\boldsymbol{\gamma}'}$]\label{theorem-gammaprime}
For $W$ satisfying Condition \ref{cond-WR-GRG-tail}   with $\expec[W^4]<\infty$ or with $\tau \in (3,5)$,
	\be
	\boldsymbol{\gamma'}=1.
	\ee
\end{theorem}

\begin{proof}
We split the proof into the two cases that cover the hypotheses of the theorem.

\vspace{0.2cm}
\noindent
{\bf  (a) Proof  under the assumption $\expec[W^4]<\infty$}.
We are now in the regime where $\beta>\beta_c$, so that $z^*_0>0$. We start from~\eqref{eq-derivative-z-B}, take the limit $B\searrow0$ and 
linearize the \col{hyperbolic} tangent. In order to control this approximation, we define $g(x)=x^2-\tanh ^2(x)$ and remark that on the basis of our \col{assumption} on $W$, we have that 
$\E[(W^2\vee 1)g(W)]<\infty$. It will be useful also to factorize  $g(x)=x^4k(x)$ with $k(x)=O(1)$ as $x\to 0$, so that we also have 
$\expec[W^6k(W)]< \infty$.
This gives
	\begin{align}\label{eq_der_zeta_zero}
	\frac{\partial z^*_0}{\partial B} 
	&= \mathbb{E}\left[\left(\ab W+\ab^{2} {W^2}\frac{\partial z^*_0}{\partial B}\right)
	\left(1-\tanh^2\left (\ab W z^*_0\right)\right) \right] \nn\\
	&= \sqrt{\expec[W]\sinh(\beta)}-e_{\sss 0} + \frac{\partial z^*_0}{\partial B}\left(\sinh(\beta)\nu 
	-\ab^4\, \expec[W^4]z_0^{*2}\right)+\nn\\
	&\qquad+\E\left[ \left (\frac{\partial z^*_0}{\partial B}\cbetaq W^2+\cbeta W\right)
	g\left( \cbeta W z_0^*\right)\right],
	\end{align}
where
	\be
	e_{\sss 0} \, = \, \left(\frac{\sinh(\beta)}{\mathbb{E}\left[W\right]}\right)^{3/2}
	\mathbb{E}\left[W^3\right]z_0^{*2}.
	\ee
Solving (\ref{eq_der_zeta_zero}) for $\frac{\partial z^*_0}{\partial B}$ gives
	\be\label{eq-approx-dz-supercrit}
	\frac{\partial z^*_0}{\partial B} = \frac{\sqrt{\expec[W]\sinh(\beta)}
	-e_{\sss 0}- \E\left[ \cbeta W g\left( \cbeta W z_0^*\right)\right]}{1-\sinh(\beta)\nu 
	+\ab^4\, \expec[W^4]z_0^{*2}
	-\E\left [\cbetaq W^2g\left( \cbeta W z_0^*\right) \right ]}.
	\ee
\col{To analyze \eqref{eq-approx-dz-supercrit} we use the lower and upper bounds in Propositions
\ref{prop-lowerbz} and \ref{prop-upperbz}
}	
Taking the limit $B\searrow 0$ in \eqref{lower_b_z} with $\bdelta=3$, $c_{1}$ given in \eqref{popi} and dividing by $z^*_0$, we obtain
	\be\label{z0astq1}
	z_0^{*2} \, \geq \, 3\, \frac{\expec[W]^2}{\expec[W^4]}\frac{1}{\sinh^2(\beta)} \left(\sinh(\beta)\nu -1\right).
	\ee
Taking  the same limit $B\searrow 0$ in \eqref{upper_b_z} and dividing by $z^*_0$ we obtain also
	\be\label{z0astq2}
	z_0^{*2} \, \leq \, 6\, \frac{\expec[W]^2}{\expec[W^4 \indic{\ab W z^* 
	\leq \atanh\frac18}]}\frac{1}{\sinh^2(\beta)} \left(\sinh(\beta)\nu -1\right).
	\ee
By Taylor expansion,
	\be\label{poldo}
	\sinh(\beta)\nu -1= \cosh(\beta_c)\nu (\beta-\beta_c)+O((\beta-\beta_c)^2),
	\ee
we conclude, from \eqref{z0astq1}, \eqref{z0astq2}, and the fact that  $\expec[W^4 \indic{ \ab W z^* \leq \atanh\frac18}] \to \expec[W^4]$ as $\beta \to \beta_c$,   that ${z_0^*}^2=O(\beta-\beta_c)$.
\col{Using this}, we can now evaluate the terms in numerator and denominator of 	\eqref{eq-approx-dz-supercrit} as  $\beta \to \beta_c$. The first term in the numerator has a non vanishing finite limit, while $e_0=O(\beta-\beta_c)$. The third term 
(\col{ignoring} the irrelevant multiplicative  factor   $ \cbeta$ ) is $ \E\left[  W g\left( \cbeta W z_0^*\right)\right]=\ab^4 z_0^{*4}\, \E \left[W^5k\left (\cbeta W z_0^*\right )\right]=O((\beta-\beta_c))$. Indeed, since $k(x)\le \frac {1}{x^{2}}$, 
	\eqn{
	\ab^4 z_0^{*4}\, \E \left[W^5k\left (\cbeta W z_0^*\right )\right] 
	\le  \cbetaq \E \left[W^3\right]  z_0^{*2}=O(\beta-\beta_c).
	}
Let us now consider the denominator and define
	\be
	D(\beta):= 1-\sinh(\beta)\nu +\ab^4\, \expec[W^4]z_0^{*2}.
	\ee
By \eqref{z0astq1},  \eqref{z0astq2} and  \eqref{poldo},
	\be
	2 \cosh(\beta_c)\nu (\beta - \beta_c)
	+O((\beta - \beta_c)^2)\le D(\beta) \le (\col{a(\beta)}-1)\cosh(\beta_c) \nu (\beta - \beta_c)
	+O((\beta - \beta_c)^2),
	\ee
where \col{$a(\beta )$} is a function that converges to $6$  as $\beta \to \betacan$.  Thus, from  the previous display we obtain $D(\beta)=O(\beta-\beta_c)$. The fourth term in the denominator of \eqref{eq-approx-dz-supercrit}, again discarding an irrelevant factor and arguing as before, is $E\left [W^2g\left( \cbeta W z_0^*\right) \right ]=\ab^4 z_0^{*4}\, \E \left[W^6k\left (\cbeta W z_0^*\right )\right]=O((\beta-\beta_c)^2)$. Therefore, summarizing our findings,
	\be\label{ord_derz0}
	\frac{\partial z^*_0}{\partial B} = O((\beta-\beta_c)^{-1}).
	\ee

From \eqref{eq-chiexplicit}, the upper bound follows using \eqref{ord_derz0}:
	\be\label{ubound_chi}
	\wdt{\chi}(\beta,0)  \, \leq \, \mathbb{E}\left[\left(1+\ab W\frac{\partial z_0^*}{\partial B}\right) \right] 
	\, \leq \, 1+ \sqrt{\sinh(\beta)\expec[W]}O((\beta-\beta_c)^{-1}).
	\ee
Similarly, for the lower bound we use that $1-\tanh^2(x) \geq 1- x^2 \,$ for every $x$, we obtain
	\begin{align}\label{lbound_chi}
	\wdt{\chi}(\beta,0)  \, 
	&\geq \, \mathbb{E}\left[\left(1+\ab W\frac{\partial z_0^*}{\partial B}\right)
	\left( 1- \ab^{2} W^2 z_0^{*2}\right) \right]\nn  \\ 
	& = 1 + \mathbb{E}\left[\ab W\frac{\partial z_0^*}{\partial B}\right] 
	- \mathbb{E}\left[\ab^{2} W^2 z_0^{*2}\right]- \mathbb{E}\left[\ab^{3}W^3z_0^{*2}
	\frac{\partial z_0^*}{\partial B}\right]\nn\\
	&= 1+ \sqrt{\sinh(\beta)\expec[W]}O((\beta-\beta_c)^{-1})
	- \sinh(\beta) \nu O(\beta-\beta_c) - \ab^{3}\expec[W^3]O(1),
	\end{align}
again starting from \eqref{eq-chiexplicit}, using \eqref{ord_derz0} and ${z_0^*}^2=O(\beta-\beta_c)$. From \eqref{ubound_chi} and  \eqref{lbound_chi} we obtain
\be\label{limchi}
0 < \liminf_{\beta \searrow \betacan}\wdt{\chi}(\beta,0^+) (\beta-\betacan) \le  \limsup_{\beta \searrow \betacan}\wdt{\chi}(\beta,0^+) (\beta-\betacan)  < \infty,
\ee
proving the theorem in the case that $\expec[W^4]<\infty$.

\vspace{0.2cm}
\noindent
{\bf (b) Proof  for  $W$ 
\col{satisfying} Condition \ref{cond-WR-GRG-tail} (ii).}
Now we generalize the previous proof in order to encompass also the case of those  $W$ whose distribution function  $F(w)=1-\pr(W>w)$ satisfies Condition \ref{cond-WR-GRG-tail}(ii). 
\col{We start by \co{defining}}
	\be
	h_{\beta,B,z^{*}}(w)=\tanh\left(\alpha w z^{*}+B \right)\alpha\,w-\alpha^{2}w^{2}z^{*},
	\ee
where the dependence of $\alpha$ on $\beta$ has been dropped, and \col{rewriting} \eqref{fixpGRG0} as
	\be
	\label{fixptau2}
	z^{*}=\mathbb{E}\left[h_{\beta,B,z^{*}}(W)\right]+\alpha^{2}z^{*}\expec[W^{2}].
	\ee
\col{Using} integration by parts,
	\begin{align}
	\mathbb{E}\left[h_{\beta,B,z^{*}}(W)\right]= &\int_{0}^{+\infty} h_{\beta,B,z^{*}}(w) dF(w)
	=-\int_{0}^{+\infty} h_{\beta,B,z^{*}}(w) d(1-F(w))\nonumber \\
	= & -\lim_{w\to +\infty} [h_{\beta,B,z^{*}}(w)(1-F(w)) ] \,+\, h_{\beta,B,z^{*}}(0)(1-F(0))  \nonumber \\
	& \,+\int_{0}^{+\infty} h^{\prime}_{\beta,B,z^{*}}(w) (1-F(w))dw.
	\end{align}
The boundary terms in the previous display vanish
\col{and therefore}
	\be
	\mathbb{E}\left[h_{\beta,B,z^{*}}(W)\right]=\int_{0}^{+\infty} h^{\prime}_{\beta,B,z^{*}}(w) (1-F(w))dw.
	\ee
Taking into account that the power law of Condition \ref{cond-WR-GRG-tail}(ii) holds for $w>w_{0}$,  we write the previous  integral as
	\be
	\mathbb{E}\left[h_{\beta,B,z^{*}}(W)\right]=\overline{G}(\beta,B,z^{*})+\bar{J}(\beta,B,z^{*}),
	\ee
where
	\be\label{GJbardef}
	\overline{G}(\beta,B,z^{*})
	:=\int_{0}^{w_{0}} h^{\prime}_{\beta,B,z^{*}}(w) (1-F(w))dw,\quad 
	\bar{J}(\beta,B,z^{*}):=\int_{w_{0}}^{+\infty} h^{\prime}_{\beta,B,z^{*}}(w) (1-F(w))dw.
	\ee
Therefore, \eqref{fixptau2} can be rewritten as
	\be
	z^{*}=\overline{G}(\beta,B,z^{*})+\bar{J}(\beta,B,z^{*})+\alpha^{2}z^{*}\expec[W^{2}].
	\ee
Now we take the limit  $B\searrow 0$ in the previous equation. Recalling that  $z^*_0:=\lim_{B\searrow0} z^* >0$, and since the following limits exist:
	\be
	\lim_{B\searrow 0}\overline{G}(\beta,B,z^{*})=G(\beta,z^*_0),\quad 
	\lim_{B\searrow 0}\bar{J}(\beta,B,z^{*})=J(\beta,z^*_0)
	\ee
by bounded convergence, \col{then we arrive to}
	\be\label{fixptau3}
	z_{0}^{*}=G(\beta,z^*_0)+J(\beta,z^*_0)+\alpha^{2}z_{0}^{*}\expec[W^{2}].
	\ee
In the next step we bound  $J(\beta,z^*_0)$. From  the definition of $ \bar{J}(\beta,B,z^{*})$ in \eqref{GJbardef}, and Condition \ref{cond-WR-GRG-tail}(ii),
	\be
	\label{boundJbar}
	c_{\sss W}\int_{w_{0}}^{+\infty} h^{\prime}_{\beta,B,z^{*}}(w) w^{-(\tau-1)}dw
	\le  \bar{J}(\beta,B,z^{*}) \le C_{\sss W}
	\int_{w_{0}}^{+\infty} h^{\prime}_{\beta,B,z^{*}}(w) w^{-(\tau-1)}dw.
	\ee
Applying the change of variable $y=\alpha z^{*} w$ 
leads to
	\be
	\int_{w_{0}}^{+\infty} h^{\prime}_{\beta,B,z^{*}}(w) w^{-(\tau-1)}dw
	=\alpha^{\tau-1} {z^{*}}^{\tau-2}\int_{\alpha w_{0}z^{*}}^{+\infty} 
	\left [ \tanh(y+B) -y\tanh^{2}(y+B)-y \right ]y^{-(\tau-1)} dy.
	\ee
Therefore, denoting
	\be
	\bar{I}(\beta,B,z^{*})
	:= \int_{\alpha w_{0}z^{*}}^{+\infty} \left [ \tanh(y+B) -y\tanh^{2}(y+B)-y \right ]y^{-(\tau-1)} dy,
	\ee
we can rewrite \eqref{boundJbar} as follows:
	\be\label{boundJbar2}
	c_{\sss W} \alpha^{\tau-1} {z^{*}}^{\tau-2} \bar{I}(\beta,B,z^{*})
	\le  \bar{J}(\beta,B,z^{*}) \le C_{\sss W} \alpha^{\tau-1} {z^{*}}^{\tau-2} \bar{I}(\beta,B,z^{*}).
	\ee
Since, again by bounded convergence,
	\be
	\lim_{B\searrow 0} \bar{I}(\beta,B,z^{*})
	= \int_{\alpha w_{0}z_{0}^{*}}^{+\infty} \left [ \tanh(y) -y\tanh^{2}(y)-y \right ]y^{-(\tau-1)} dy 
	=: I(\beta,z_{0}^{*}),
	\ee
we obtain from \eqref{boundJbar} that
	\be\label{fiordiligi}
	c_{\sss W} \alpha^{\tau-1} {z_{0}^{*}}^{\tau-2} {I}(\beta,z_{0}^{*})
	\le  {J}(\beta,z_{0}^{*}) \le C_{\sss W} \alpha^{\tau-1} {z_{0}^{*}}^{\tau-2} {I}(\beta,z_{0}^{*}).
	\ee
On the other hand,  since $\tanh(y) -y\tanh^{2}(y)-y<0$ for $y>0$, \col{we also have}
	\begin{align}\label{boundI}
	k(\tau)&:=\int_{1}^{+\infty} [y\tanh^{2}(y)+y- \tanh(y)] y^{-(\tau-1)} dy\le -I(\beta,z_{0}^{*}) \nonumber \\
	&\le \int_{0}^{+\infty} [y\tanh^{2}(y)+y- \tanh(y)] y^{-(\tau-1)} dy=:K(\tau).
	\end{align}
Therefore, from \eqref{fixptau3}, \eqref{fiordiligi} and \eqref{boundI},
	\be\label{ugo1}
\col{	
        {z_{0}^{*}} \ge 
	G(\beta,z^*_0) -  c_{W} \alpha^{\tau-1} {z_{0}^{*}}^{\tau-2} K(\tau) 
	+\alpha^{2}z_{0}^{*}\expec[W^{2}]
}	
	\ee
and 
	\be\label{ugo2}
\col{
	{z_{0}^{*}} \le 
	G(\beta,z^*_0) - C_{W} \alpha^{\tau-1} {z_{0}^{*}}^{\tau-2} k(\tau) 
	+\alpha^{2}z_{0}^{*}\expec[W^{2}].
}	
	\ee
The next step is to control the behaviour of $G(\beta,z^*_0)$ as $\beta\to \beta_{c}$. We start by showing that $G(\beta,z^*_0)$ is $O(z_{0}^{*3})$ as $\beta\to \beta_{c}$.  From the definition of  $G(\beta,z^*_0)$,
	\be
	G(\beta,z^*_0)=\int_{0}^{w_{0}} \left [ -\alpha^{2}z_{0}^{*}w 
	- \tanh^{2}(\alpha z_{0}^{*}w )\alpha^{2}z_{0}^{*}w +\alpha \tanh(\alpha z_{0}^{*}w )\right ](1-F(w))dw.
	\ee
Since the function  between the square brackets is negative for $y>0$ and decreasing, we have
	\begin{align}
	0\ge& G(\beta,z^*_0)
	\ge  [-\alpha^{2} w_{0}z_{0}^{*}-\alpha^{2} w_{0}z_{0}^{*} \tanh^{2}(\alpha w_{0}z_{0}^{*})
	+\alpha \tanh(\alpha w_{0}z_{0}^{*}) ]   \int_{0}^{ w_{0}} (1-F(w)) dw\nonumber  \\
	&\ge   [-\alpha^{2} w_{0}z_{0}^{*}-\alpha^{2} w_{0}z_{0}^{*} \tanh^{2}(\alpha w_{0}z_{0}^{*})
	+\alpha \tanh(\alpha w_{0}z_{0}^{*}) ] =-\frac{4}{3}\alpha^{4}w^{3}z_{0}^{*3}+O(z_{0}^{*5})
	\end{align}
where the last equality is obtained by Taylor expansion. 

Thus, the previous inequality implies that $G(\beta,z^*_0)=O({z^{*}_{0}}^{3})$. Again, from \eqref{ugo1} and \eqref{ugo2} dividing by $z_{0}^{*}$,
	\be
	1-\alpha^{2}\E[W^{2}]  \ge z^{* \tau-3}_{0}\big (G(\beta,z^*_0)   z^{* 2-\tau}_{0}  
	-  c_{\sss W} \alpha^{\tau-1} K(\tau) \big),
	\ee
and 
	\be
	1-\alpha^{2}\E[W^{2}]  \le z^{* \tau-3}_{0}\big (G(\beta,z^*_0)   z^{* 2-\tau}_{0}  
	-  C_{\sss W} \alpha^{\tau-1} k(\tau) \big).
	\ee
Since $G(\beta,z^*_0)   z^{* 2-\tau}_{0}=O(z^{* 5-\tau}_{0})$ and $\tau\in (3,5)$, the previous inequalities  together with \eqref{poldo} imply  that $z^{* \tau-3}_{0}=O(\beta-\beta_{c})$ as $\beta \searrow \beta_{c}$.  

Next, we consider the 
derivative of $z^{*}_{0}$. Again, taking the limit $B\searrow0$ for $\beta>\betacan$ of \eqref{eq-derivative-z-B} we obtain
	\be\label{leonore2}
	\frac{\partial z^*_0}{\partial B}
	=\frac{\alpha \E[W] -\alpha \E[W \tanh^{2}(\alpha W z^{*}_{0})]}{1-\alpha^{2}\E[W^{2}] 
	+ \alpha^{2} \E[W^{2} \tanh^{2}(\alpha W z^{*}_{0})]}.
	\ee
Since the numerator has a finite positive limit as $\beta \searrow \beta_{c}$ (in particular, the second term is vanishing), we will focus on the denominator
	\be\label{giara}
	D_{2}(\beta):=1-\alpha^{2}\E[W^{2}] + \alpha^{2} \E[W^{2} \tanh^{2}(\alpha W z^{*}_{0})].
	\ee
We start by decomposing the average 
	\be\label{uno}
	\E[W^{2} \tanh^{2}(\alpha W z^{*}_{0})]=\E[W^{2} \tanh^{2}(\alpha W z^{*}_{0}) \indic{W\leq w_{0}}]
	+\E[W^{2} \tanh^{2}(\alpha W z^{*}_{0})\indic{W>w_{0}}],
	\ee
and analyze the two terms separately. The first one can be bounded as follows
	\be\label{nessuno}
	0\le \E[W^{2} \tanh^{2}(\alpha W z^{*}_{0}) \indic{W\leq w_{0}}]\le \alpha^{2} w_{0}^{4} z_{0}^{*2}
	\ee
showing that
	\be\label{centomila}
	\E[W^{2} \tanh^{2}(\alpha W z^{*}_{0}) \indic{W\leq w_{0}}]=O(z_{0}^{*2})
	=O((\beta-\beta_{c})^{\frac{2}{\tau-3}}),
	\ee 
with the exponent satisfying $2/(\tau-3)>1$ since $\tau \in (3,5)$. The second term  can be treated  with the integration by parts formula
	\begin{align}\label{naso}
	&\E[W^{2} \tanh^{2}(\alpha W z^{*}_{0})\indic{W>w_{0}}]
	=-\lim_{w\to +\infty}  [w^{2} \tanh^{2}(\alpha w z^{*}_{0})(1-F(w))]\nonumber \\
	&+ w_{0}^{2} \tanh^{2}(\alpha w_{0} z^{*}_{0})(1-F(w_{0}))
	+ \int_{w_{0}}^{+\infty} \frac{\partial}{\partial w} [w^{2} \tanh^{2}(\alpha w z^{*}_{0})](1-F(w))dw .
	\end{align}
Since $\tau>3$, from Condition \ref{cond-WR-GRG-tail} we conclude that the limit in the previous display vanishes. It is also simple to see that
	\be \label{fu}
	w_{0}^{2} \tanh^{2}(\alpha w_{0} z^{*}_{0})(1-F(w_{0}))=O(z_{0}^{*2})=O((\beta-\beta_{c})^{\frac{2}{\tau-3}}).
	\ee
From \eqref{giara} and using \eqref{uno}, 
\eqref{centomila}, \eqref{naso}, \eqref{fu}, we can write
	\be\label{guglielmo}
	D_{2}(\beta)=\bar{D}_{2}(\beta)+O((\beta-\beta_{c})^{\frac{2}{\tau-3}}),
	\ee
with
	\be
	\bar{D}_{2}(\beta):= (1-\alpha^{2}\E[W^{2}] )
	+\alpha^{2}  \int_{w_{0}}^{+\infty} \frac{\partial}{\partial w} [w^{2} \tanh^{2}(\alpha w z^{*}_{0})](1-F(w))dw .
	\ee
The second term in the r.h.s.\ of \eqref{guglielmo}  is $O((\beta-\beta_{c})^{s})$ with $s>1$, therefore we can forget it  since the first  \col{term} of   $\bar{D}_{2}(\beta)$ is  $O(\beta-\beta_{c})$.  
Now we focus on the second term of $\bar{D}_{2}(\beta)$. 

By  using \eqref{eq-power-law} and applying the change of variable $y=\alpha z^{*} w$, we can bound the integral in the last display as
	\be\label{mattia}
 	\int_{w_{0}}^{+\infty} \frac{\partial}{\partial w} [w^{2} \tanh^{2}(\alpha w z^{*}_{0})](1-F(w))dw
	\le C_{\sss W} \alpha^{\tau-3}z_{0}^{*\tau-3} M(\tau),
	\ee
where
	\be\label{pascal}
	M(\tau):= \int_{0}^{+\infty} \left [2y \tanh^{2}(y)+2y^{2}\tanh(y)(1-\tanh^{2}(y)) 	\right ]y^{-(\tau-1)} dy,
	\ee
and the bound in \eqref{mattia} is obtained thanks to the positivity of the integrand. The convergence of the integral is ensured by the fact that this function is $O(y^{4-\tau})$ close to $y=0$ with $1 > 4-\tau > -1$ and 
is $O(y^{-\tau +2})$ as $y\to \infty$ with $-\tau +2<-1$. In a similar fashion, we can also obtain
	\be\label{vitangelo}
	 \int_{w_{0}}^{+\infty} \frac{\partial}{\partial w} [w^{2} \tanh^{2}(\alpha w z^{*}_{0})](1-F(w))dw\ge c_{\sss W} 	\alpha^{\tau-3}z_{0}^{*\tau-3} m(\tau),
	\ee
with 
	\be\label{moscarda}
	m(\tau):= \int_{\varepsilon}^{+\infty} \left [2y \tanh^{2}(y)+2y^{2}\tanh(y)(1-\tanh^{2}(y)) 	\right ]y^{-(\tau-1)} dy,
	\ee
for $\beta$ sufficiently close to $\beta_{c}$. At this stage $\varepsilon>0$ is an arbitrary  fixed quantity that will be chosen later (but independently of $\beta$).  
By  \eqref{ugo1} and \eqref{ugo2},
	\be
	\frac{G(\beta,z_{0}^{*}){ z_{0}^{*}}^{-1} -  (1-\alpha^{2}\E[W^{2}] )}{c_{\sss W} \alpha^{\tau-1} K(\tau)}
	\le z_{0}^{*\tau-3}\le \frac{G(\beta,z_{0}^{*}){ z_{0}^{*}}^{-1} 
	-  (1-\alpha^{2}\E[W^{2}] )}{C_{\sss W} \alpha^{\tau-1} k(\tau)},
	\ee
which,  substituted in  \eqref{mattia} and \eqref{vitangelo}, gives
	\begin{align}
	{G(\beta,z_{0}^{*}){ z_{0}^{*}}^{-1}} \frac{m(\tau)}{K(\tau)}-{(1-\alpha^{2}\E[W^{2}] )} \frac{m(\tau)}{K(\tau)}
	&\le \alpha^{2}  \int_{w_{0}}^{+\infty} \frac{\partial}{\partial w} [w^{2} \tanh^{2}(\alpha w z^{*}_{0})](1-F(w))dw \nonumber \\
	&\le G(\beta,z_{0}^{*}){ z_{0}^{*}}^{-1}\frac{M(\tau)}{k(\tau)}-{(1-\alpha^{2}\E[W^{2}] )} \frac{M(\tau)}{k(\tau)}.
	\end{align}
By definition of $\bar{D}_{2}(\beta)$,
	\be\label{bound-D}
	{G(\beta,z_{0}^{*}){ z_{0}^{*}}^{-1}} \frac{m(\tau)}{K(\tau)}+ (1-\alpha^{2}\E[W^{2}] ) \left( 1- \frac{m(\tau)}{K(\tau)}\right)
	\le \bar{D}_{2}(\beta)\\
	\le  {G(\beta,z_{0}^{*}){ z_{0}^{*}}^{-1}} \frac{M(\tau)}{k(\tau)}
	+ (1-\alpha^{2}\E[W^{2}] ) \left( 1- \frac{M(\tau)}{k(\tau)}\right).\nn
	\ee
In the last step of the proof, we show that  $\frac{m(\tau)}{K(\tau)}>1$.  This can be done by properly choosing the arbitrary quantity $\varepsilon$ in  \eqref{moscarda}. We will prove the first inequality, the second one can be obtained in the same way.   Starting  from \eqref{boundI} and \eqref{moscarda},  we introduce the functions $K_{b}(\tau)$ and $m_{a}(\tau)$ for $a\ge 0,b\ge 0$ as
	\be
	K_{b}(\tau):= \int_{b}^{+\infty}\frac{d}{dy}  \left [y^{2} - y\tanh (y) \right ] y^{-(\tau-1)} dy,\qquad  
	m_{a}(\tau) :=  \int_{a}^{+\infty}  \frac{d}{dy} \left [y^{2} \tanh^{2} (y) \right ] y^{-(\tau-1)} dy,
	\ee
which coincide with $K(\tau)$ and $m(\tau)$ for $b=0$ and $a=\varepsilon$, respectively. By applying  the integration by parts formula the two functions  can be written as
 	\begin{align}
	&K_{b}(\tau)= -b^{3-\tau} + b^{2-\tau} \tanh(b) +(\tau-1)\int_{b}^{+\infty} (y^{2}-y\tanh(y)) y^{-\tau}dy,\label{humpty}\\
	&  m_{a}(\tau) =-a\tanh(a)+(\tau-1) \int_{a}^{+\infty} y^{2}\tanh^{2}(y) y^{-\tau} dy.\label{dumpty}
 	\end{align}
Since 
	\be
	\lim_{a\to 0^{+}\atop b\to 0^{+}} \frac{m_{a}(\tau)}{K_{b}(\tau)}=\frac{m_{0}(\tau)}{K(\tau)}
	=\frac{\int_{0}^{+\infty} y^{2}\tanh^{2}(y)\, y^{-\tau} dy}{\int_{0}^{+\infty}  \left [y^{2}-y\tanh(y) \right] \, y^{-\tau}  dy} >1,
	\ee
where the inequality can be proved by observing that $y^{2}\tanh^{2}(y) > y^{2}-y\tanh(y)$ for all $y>0$,  \col{then} for any $\varepsilon>0$ sufficiently small,
	\be
	\frac{m(\tau)}{K(\tau)}>1.
	\ee
Since $G(\beta,z_{0}^{*}){ z_{0}^{*}}^{-1}=O(z_{0}^{*2})=O((\beta-\beta_{c})^{{s}})$ with $s=\frac{2}{\tau-3}>1$ and $ (1-\alpha^{2}\E[W^{2}] )=O(\beta-\beta_{c})$, with $1-\alpha^{2}\E[W^{2}]<0$ for $\beta>\beta_{c}$ and close to $\beta_{c}$ (see \eqref{poldo}),
 we conclude that $0<\bar{D}_{2}(\beta)=O(\beta-\beta_{c})$, for the same values of $\beta$. This proves that
 	\be\label{pepe}
	0<\frac{\partial z^*_0}{\partial B} = O((\beta-\beta_c)^{-1}).
	\ee
The previous equation together with  ${z^{*}_0}^{\tau-3}=O(\beta-\beta_{c})$ allows us to conclude the proof along the same lines of the case with $\E[W^{4}]<\infty$. Indeed, the upper bound \eqref{ubound_chi} is still valid in the present case, since only the first moment of $W$ is involved.  For the lower bound we argue as follows. Since, $1-\tanh^{2}(x)>1-\tanh(x)>1-x$ for $x>0$, we have
	\begin{align}\label{pizarro}
	\wdt{\chi}(\beta,0)  \, 
	&\geq \, \mathbb{E}\left[\left(1+\ab W\frac{\partial z_0^*}{\partial B}\right)\left( 1- \ab W z_0^{*}\right) \right]\nn \\
	& = 1- \sqrt{\sinh(\beta)\expec[W]} z^{*}_{0}
	+\sqrt{\sinh(\beta)\expec[W]}  \frac{\partial z_0^*}{\partial B} - \sinh(\beta) \nu z^{*}_{0}  \frac{\partial z_0^*}{\partial B} \nn \\
	& = 1- \sqrt{\sinh(\beta)\expec[W]} O((\beta-\beta_c)^{1/(\tau-3)})  
	+ \sqrt{\sinh(\beta)\expec[W]}  O((\beta-\beta_c)^{-1})\nn \\
	&  -  \sinh(\beta) \nu   O((\beta-\beta_c)^{1/(\tau-3)})   O((\beta-\beta_c)^{-1}) .
	\end{align}
The inequalities \eqref{ubound_chi} and \eqref{pizarro} imply \eqref{limchi} concluding the proof of the theorem.
\end{proof}

\section{Non-classical limit theorems at criticality: proof of Theorem \ref{thm-nonclassicalclt} }\label{sec-NC-CLT}
In this section we prove Theorem~\ref{thm-nonclassicalclt}. For this, we follow the strategy of the proof for the Curie-Weiss model 
(\col{see e.g. 
\cite[Theorem~V.9.5]{El}}). 
\col{It} suffices to prove that for any real number $r$
	\be
	\lim_{N\rightarrow\infty} P_{\sss N}
	\left( \exp\left(r \frac{S_{\sss N}}{N^{\boldsymbol{\delta}/(\boldsymbol{\delta}+1)}}\right) \right) 
	= \frac{\int_{-\infty}^{\infty} \exp\left(rz - f(z)\right)\dint z}{\int_{-\infty}^{\infty} \exp\left(- f(z)\right)\dint z}.
	\ee

As observed in~\cite{GGvdHP}, the measure $P_{\sss N}$ is approximately equal to the inhomogeneous Curie-Weiss measure
	\be\label{eq-ptildemeasure}
	\widetilde{P}_{\sss N}(\col{g})  
	=
	\frac{1}{\widetilde{Z}_{\sss N}}\sum_{\sigma \in \Omega_{\sss N}}\col{g}(\sigma)\e^{\frac{1}{2}\sinh \beta
	\sum_{i,j \in[N]}\frac{w_i w_j}{\ell_{\sss N}}\sigma_i \sigma_j}
	=\frac{1}{\widetilde{Z}_{\sss N}}  \sum_{\sigma \in \Omega_{\sss N}} \col{g}(\sigma)
	\e^{\frac{1}{2}\frac{\sinh \beta}{\ell_{\sss N}}\left(\sum_{i\in[N]}w_i \sigma_i\right)^2},
	\ee
where $\col{g}(\sigma)$ is any bounded function defined in $\Omega_{\sss N}$ and $\widetilde{Z}_{\sss N}$ is the associated normalization factor, i.e.,
	\be\label{eq-Ztildedef}
	\widetilde{Z}_{\sss N} = \sum_{\sigma \in \Omega_{\sss N}} 
	\e^{\frac{1}{2}\frac{\sinh \beta}{\ell_{\sss N}}\left(\sum_{i\in[N]}w_i \sigma_i\right)^2}.
	\ee
We first prove the theorem for this measure $\widetilde{P}_{\sss N}$, which is the rank-1 inhomogeneous Curie-Weiss model with $\beta$ replaced with $\sinh(\beta)$. 

For this, we use the Hubbard-Stratonovich identity to rewrite $\widetilde{P}_{\sss N}\Big( \exp\Big(r \frac{S_{\sss N}}{N^{\lambda}}\Big)\Big)$ as a fraction of two integrals of an exponential function in Lemma~\ref{pippo} \change{in Section \ref{sec-rewrite-MGF}}. We next split the analysis into the cases $\expec[W^4]<\infty$ and $\tau\in(3,5)$ in Sections~\ref{sec-convergence-EW4finite} and~\ref{sec-convergence-tau35}, respectively. For both these cases we analyze the exponents in the integrals and use Taylor expansions to show that they converge in Lemmas~\ref{pluto} and~\ref{lem-plutotau35}, respectively. We then use dominated convergence to show that the integrals also converge in Lemmas~\ref{lem-integralconvE4} and~\ref{lem-integralconvtau35}, respectively. The tail behavior of $f(x)$ for $\tau\in(3,5)$ is analyzed in Lemma~\ref{lem-largeztau35}. 
\col{Combining these results we conclude the proof of Theorem \ref{thm-nonclassicalclt} in Section \ref{final-proof}: we first prove the theorem for 
 $\widetilde{P}_{\sss N}$ and then we show that the theorem also holds for $P_{\sss N}$ in Lemma~\ref{lem-fromtildetoreal}.
Finally,} in Section~\ref{sec-scaling-window}, we discuss how to \change{adapt} the proof to obtain the results on the scaling window.


\subsection{Rewrite of the moment generating function}\label{sec-rewrite-MGF}
To ease the notation we first rescale $S_{\sss N}$ by $N^\lambda$ and later set $\lambda = \boldsymbol{\delta}/(\boldsymbol{\delta}+1)$. We rewrite $\widetilde{P}_{\sss N}\left( \exp\left(r \frac{S_{\sss N}}{N^{\lambda}}\right)\right)$ in the following lemma:
\begin{lemma}[Moment generating function of $S_{\sss N}/N^{\lambda}$]
\label{pippo}
For $B=0$,
	\be
	\widetilde{P}_{\sss N}\Big( \exp\Big(r \frac{S_{\sss N}}{N^{\lambda}}\Big)\Big)
	= \frac{\int_{-\infty}^{\infty} \e^{-N G_{\sss N}(z;r)}\dint z}{\int_{-\infty}^{\infty} \e^{-N G_{\sss N}(z;0)}\dint z},
	\ee
where
	\be \label{eq-defGNzr}
	G_{\sss N}(z;r) = \frac12 z^2 - \expec\left[ \log \cosh \left( \alpha_{\sss N}(\beta)W_{\sss N} z 
	+ \frac{r}{N^{\lambda}}\right) \right],
	\ee
with
	\be
	\alpha_{\sss N}(\beta) = \sqrt{\frac{\sinh \beta}{\expec[W_{\sss N}]}}.
	\ee 
\end{lemma}

\begin{proof}
We use the Hubbard-Stratonovich identity, i.e., we write $\e^{t^2/2} = \mathbb{E} \left[\e^{tZ}\right]$, with $Z$ standard Gaussian, to obtain
	\begin{align}
	\widetilde{Z}_{\sss N} \widetilde{P}_{\sss N}\Big( \exp\Big(r \frac{S_{\sss N}}{N^{\lambda}}\Big)\Big) 
	&=  \sum_{\sigma \in \Omega_{\sss N}} \e^{\frac{r}{N^{\lambda}}\sum_{i\in[N]} \sigma_i}
	\e^{\frac{1}{2}\frac{\sinh \beta}{\ell_{\sss N}}\left(\sum_{i\in[N]}w_i \sigma_i\right)^2} \nn\\
	&=  \sum_{\sigma \in \Omega_{\sss N}} \e^{\frac{r}{N^{\lambda}}\sum_{i\in[N]} \sigma_i} \expec\Big[
	\e^{\sqrt{\frac{\sinh \beta}{\ell_{\sss N}}}\left(\sum_{i\in[N]}w_i \sigma_i\right)Z}\Big]\\
	&= 2^N \expec\Big[\prod_{i\in[N]} \cosh\Big(\sqrt{\frac{\sinh \beta}{\ell_{\sss N}}}w_i Z+\frac{r}{N^\lambda}\Big)\Big] 
	= 2^N \expec\Big[\e^{\sum_{i\in[N]} \log\cosh\Big(\sqrt{\frac{\sinh \beta}{\ell_{\sss N}}}w_i Z+\frac{r}{N^\lambda}\Big)}\Big].		\nn
	\end{align}
We rewrite the sum in the exponential, using the fact that $\col{W_{\sss N}=w_{U_{\sss N}}}$, where $\col{U_{\sss N}}$ is a uniformly chosen vertex in $[N]$, as
	\begin{align}
	\widetilde{Z}_{\sss N} \widetilde{P}_{\sss N}\Big( \exp\Big(r \frac{S_{\sss N}}{N^{\lambda}}\Big)\Big) 
	&= 2^N \expec\Big[\exp{\Big\{N \mathbb{E}\ \Big[\log \cosh\Big(\sqrt{\frac{\sinh \beta}
	{N \mathbb{E}[W_{\sss N}]}}W_{\sss N}Z + \frac{r}{N^\lambda}\Big)
	\mid\, Z \,\Big]\Big\}}\Big] \nn\\
	&=\frac{2^N }{\sqrt{2\pi}}\int_{-\infty}^{\infty} \e^{-z^2/2}
	\exp{\Big\{N \mathbb{E}\ \Big[\log \cosh\Big(\alpha_{\sss N}(\beta)W_{\sss N}\frac{z}{\sqrt{N}} 
	+ \frac{r}{N^\lambda}\Big)\Big]\Big\}}\dint z.	
	\end{align}
By substituting $z/\sqrt{N}$ for $z$, we get
	\begin{align}\label{eq-substitutesqrtN}
	\widetilde{Z}_{\sss N} \widetilde{P}_{\sss N}\Big( \exp\Big(r \frac{S_{\sss N}}{N^{\lambda}}\Big)\Big) 
	&= 2^N\sqrt{\frac{N  }{2\pi}}\int_{-\infty}^{\infty} \e^{-Nz^2/2}
	\exp{\Big\{N \mathbb{E}\ \Big[\log \cosh\Big(\alpha_{\sss N}(\beta)W_{\sss N}z 
	+ \frac{r}{N^\lambda}\Big)\Big]\Big\}}\dint z \nn\\
	&= 2^N\sqrt{\frac{N}{2\pi}}\int_{-\infty}^{\infty} \e^{-N G_{\sss N}(z;r)}\dint z.
	\end{align}
In a similar way we can rewrite
	\be
	\widetilde{Z}_{\sss N} = 2^N\sqrt{\frac{N}{2\pi}}\int_{-\infty}^{\infty} \e^{-N G_{\sss N}(z;0)}\dint z,
	\ee
so that the lemma follows.
\end{proof}


\subsection{Convergence for $\expec[W^4]<\infty$}\label{sec-convergence-EW4finite}
We analyze the asymptotics of the function $G_{\sss N}(z;r)$:
\begin{lemma}[Asymptotics of $G_{\sss N}$ for ${\expec[W^4]}<\infty$]
\label{pluto}
For $\beta=\beta_{c,N}$, $B=0$ and $\expec[W^4]<\infty$,
	\be
	\lim_{N\rightarrow\infty} N G_{\sss N}(z/N^{1/4};r) 
	= -z r \sqrt{\frac{\expec[W]}{\nu}}+\frac{1}{12}\frac{\expec[W^4]}{\expec[W^2]^2}z^4.
	\ee
\end{lemma}

\begin{proof}
Taylor expanding $\log\cosh(x)$ about $x=0$ gives that
	\be
	\label{expand}
	\log \cosh(x) = \frac{x^2}{2} -\frac{1}{12} x^4 + O(x^6).
	\ee
We want to use this to analyze $ N G_{\sss N}(z/N^{1/4};r)$ and hence need to analyze the second, fourth and sixth moment of $\sqrt{\frac{\sinh \beta_{c,N}}{\mathbb{E}[W_{\sss N}]}}W_{\sss N}\frac{z}{N^{1/4}} + \frac{r}{N^\lambda}$.

The second moment equals, using that $\lambda=\boldsymbol{\delta}/(\boldsymbol{\delta}+1)=3/4$,
	\begin{align}\label{eq-2ndmomentargG}
	\mathbb{E}\Big[\Big(\alpha_{\sss N}(\beta_{c,N})W_{\sss N}\frac{z}{N^{1/4}} + \frac{r}{N^\lambda}\Big)^2\Big] 
	&= \sinh \beta_{c,N} \nu_{\sss N} \frac{z^2}{\sqrt{N}} 
	+ 2\sqrt{\sinh \beta_{c,N}\mathbb{E}[W_{\sss N}]}\frac{z r}{N} + \frac{r^2}{N^{6/4}} \nn\\
	&= \frac{z^2}{\sqrt{N}} +2\frac{z r}{N} \sqrt{\frac{\expec[W_{\sss N}]}{\nu_{\sss N}}}+o(1/N),
	\end{align}
where we have used that $\sinh \beta_{c,N}=1/\nu_{\sss N}$ in the second equality.

For the fourth moment we use that by assumption the first four moments of $W_{\sss N}$ are $O(1)$. Hence, for all $r$,
	\begin{align}
	\mathbb{E}\Big[\Big(\alpha_{\sss N}(\beta_{c,N})W_{\sss N}\frac{z}{N^{1/4}} + \frac{r}{N^\lambda}\Big)^4\Big]
	&=\frac{\sinh^2\beta_{c,N}}{\expec[W_{\sss N}]^2}\expec[W_{\sss N}^4]\frac{z^4}{N}
	+O\Big(\frac{1}{N^{3/4+\lambda}}+\frac{1}{N^{2/4+2\lambda}}+\frac{1}{N^{1/4+3\lambda}}+\frac{1}{N^{4\lambda}}\Big)\nn\\
	&=\frac{\expec[W_{\sss N}^4]}{\expec[W_{\sss N}^2]^2}\frac{z^4}{N}+o(1/N).
	\end{align}

For the sixth moment, we have to be a bit more careful since $\expec[W^6]$ is potentially infinite. We can, however, use that
	\be
	\mathbb{E}[W_{\sss N}^6] = \frac{1}{N} \sum_{i=1}^N w_i^6 
	\leq (\max_{i=1}^N w_i)^2 \frac{1}{N}\sum_{i=1}^N w_i^4 = (\max_i w_i)^2 \mathbb{E}[W_{\sss N}^4].
	\ee
It can easily be seen that $\max_{i=1}^N w_i = o(N^{1/4})$ when $W_{\sss N}\stackrel{{\mathcal D}}{\longrightarrow} W$ and $\expec[W_{\sss N}^4]\to \expec[W^4]<\infty$. Hence,
	\be
	\mathbb{E}\Big[\Big(\alpha_{\sss N}(\beta_{c,N})W_{\sss N}\frac{z}{N^{1/4}}\Big)^6\Big] 
	= \frac{\sinh^3\beta_{c,N}}{\expec[W_{\sss N}]^3}\expec[W_{\sss N}^6]\frac{z^6}{N^{6/4}} 
	= \frac{o(N^{1/2})\expec[W_{\sss N}^4]}{\expec[W_{\sss N}^2]^3} \frac{z^6}{N^{6/4}} = o(1/N).
	\ee
In a similar way, it can be shown that
	\be
	\mathbb{E}\Big[\Big(\alpha_{\sss N}(\beta_{c,N})W_{\sss N}\frac{z}{N^{1/4}} + \frac{r}{N^\lambda}\Big)^6\Big] = o(1/N).
	\ee
Putting everything together and using that the first four moments of $W_{\sss N}$ converge by assumption,
	\begin{align}
	\label{mulp-N-GN-conv}
	\lim_{N\rightarrow\infty} N G_{\sss N}(z/N^{1/4};r) &= \lim_{N\rightarrow\infty} 
	\left(\frac{\sqrt{N}}{2} z^2 - N\expec\left[ \log \cosh \Big(\alpha_{\sss N}(\beta_{c,N})W_{\sss N}\frac{z}{N^{1/4}}
	+ \frac{r}{N^{\lambda}}\Big) \right]\right)\nn\\
	& =-z r \sqrt{\frac{\expec[W]}{\nu}}+\frac{1}{12}\frac{\expec[W^4]}{\expec[W^2]^2}z^4.
	\end{align}
\end{proof}
\medskip

From Lemma \ref{pluto} it also follows that the integral converges as we show next:

\begin{lemma}[Convergence of the integral for ${\expec[W^4]}<\infty$]\label{lem-integralconvE4}
For $\beta=\beta_{c,N}$, $B=0$ and $\expec[W^4]<\infty$,
	\be
	\lim_{N\rightarrow\infty}\int_{-\infty}^\infty \e^{- N G_{\sss N}(z/N^{1/4};r)} \dint z 
	= \int_{-\infty}^\infty \e^{z r \frac{\expec[W]}{\sqrt{\expec[W^2]}}-\frac{1}{12}\frac{\expec[W^4]}{\expec[W^2]^2}z^4} \dint z.
	\ee
\end{lemma}

\begin{proof}
We prove this lemma using dominated convergence. Hence, we need to find a lower bound on $N G_{\sss N}(z/N^{1/4};r)$. We first rewrite this function by using that
	\be
	\expec\left[ \frac12 \alpha_{\sss N}(\beta_{c,N})^2 W_{\sss N}^2\left(\frac{z}{N^{1/4}}\right)^2 \right] 
	= \frac{1}{2} \left(\frac{z}{N^{1/4}}\right)^2.
	\ee
Hence,
	\begin{align}\label{eq-rewriteGN}
	G_{\sss N}(z/N^{1/4};r)
	& = \expec\left[ \frac12 \alpha_{\sss N}(\beta_{c,N})^2 W_{\sss N}^2\left(\frac{z}{N^{1/4}}\right)^2 
	-\log \cosh \Big(\alpha_{\sss N}(\beta_{c,N})W_{\sss N}\frac{z}{N^{1/4}}+ \frac{r}{N^{\lambda}}\Big) \right] \nn\\
	&= \expec\left[ \frac12 \left(\alpha_{\sss N}(\beta_{c,N}) W_{\sss N}\frac{z}{N^{1/4}}\right)^2 
	-\log \cosh \Big(\alpha_{\sss N}(\beta_{c,N})W_{\sss N}\frac{z}{N^{1/4}}\Big) \right] \\
	&\qquad- \expec\left[ \log \cosh \Big(\alpha_{\sss N}(\beta_{c,N})W_{\sss N}\frac{z}{N^{1/4}}
	+ \frac{r}{N^{\lambda}}\Big) -\log \cosh \Big(\alpha_{\sss N}(\beta_{c,N})W_{\sss N}\frac{z}{N^{1/4}}\Big)  \right]\nn
	\end{align}
Since
	\be
	\frac{\dint^2}{\dint x^2} (\frac12 x^2 - \log \cosh x ) = 1- (1-\tanh^2(x)) = \tanh^2(x)\geq 0,
	\ee
the function $\frac12 x^2 - \log \cosh x$ is convex and we can use Jensen's inequality to bound
	\begin{align}
	\expec\Big[ \frac12 &\left(\alpha_{\sss N}(\beta_{c,N}) W_{\sss N}\frac{z}{N^{1/4}}\right)^2 
	-\log \cosh \Big(\alpha_{\sss N}(\beta_{c,N})W_{\sss N}\frac{z}{N^{1/4}}\Big) \Big]  \nn\\
	&\geq \frac12 \left(\alpha_{\sss N}(\beta_{c,N})\expec[ W_{\sss N}]\frac{z}{N^{1/4}}\right)^2 
	-\log \cosh \Big(\alpha_{\sss N}(\beta_{c,N})\expec[W_{\sss N}]\frac{z}{N^{1/4}}\Big) \nn\\
	&=\frac{1}2\Big(\sqrt{\frac{\expec[ W_{\sss N}]}{\nu_{\sss N}}}\frac{z}{N^{1/4}}\Big)^2 
	-\log \cosh \Big(\sqrt{\frac{\expec[ W_{\sss N}]}{\nu_{\sss N}}}\frac{z}{N^{1/4}}\Big). 
	\end{align}
As observed in the proof of~\cite[Theorem~V.9.5]{El}, there exist positive constants $A$ and $\varepsilon$ so that
	\be\label{eq-taylorh24}
	\frac12 x^2 - \log \cosh x \geq d(x) := \left\{ \begin{array}{ll} \varepsilon x^4, 
	& \quad {\rm for\ }|x|\leq A,\\ \varepsilon x^2,&\quad {\rm for\ }|x|> A. \\  \end{array}\right.
	\ee
To bound the second term in~\eqref{eq-rewriteGN}, we can use the Taylor expansion
	\be
	\log\cosh(a+x) = \log\cosh(a) + \tanh(\xi)x,
	\ee
for some $\xi\in(a,a+x)$, and that $|\tanh(\xi)|\leq |\xi | \leq |a|+|x|$ to obtain
	\begin{align}
	\expec\Big[ \log \cosh &\Big(\alpha_{\sss N}(\beta_{c,N})W_{\sss N}\frac{z}{N^{1/4}}+ \frac{r}{N^{\lambda}}\Big) 
	-\log \cosh \Big(\alpha_{\sss N}(\beta_{c,N})W_{\sss N}\frac{z}{N^{1/4}}\Big)  \Big] \nn\\
	&\leq \expec\Big[\Big| \log \cosh \Big(\alpha_{\sss N}(\beta_{c,N})W_{\sss N}\frac{z}{N^{1/4}}+ \frac{r}{N^{\lambda}}\Big) 
	-\log \cosh \Big(\alpha_{\sss N}(\beta_{c,N})W_{\sss N}\frac{z}{N^{1/4}}\Big) \Big| \Big] \nn\\
	& \leq \expec\Big[\Big(\Big|\alpha_{\sss N}(\beta_{c,N})W_{\sss N}\frac{z}{N^{1/4}}\Big|
	+\frac{|r|}{N^{\lambda}}\Big)  \frac{|r|}{N^{\lambda}}\Big] 
	= \alpha_{\sss N}(\beta_{c,N}) \mathbb{E}[W_{\sss N}] \frac{|zr|}{N^{1/4+\lambda}}+ \frac{r^2}{N^{2\lambda}}. \nn\\
	&= \sqrt{\frac{\expec[ W_{\sss N}]}{\nu_{\sss N}}} \frac{|zr|}{N} + \frac{r^2}{N^{3/2}}.
	\end{align}
Hence,
	\be
	\e^{- N G_{\sss N}(z/N^{1/4};r)} 
	\leq \exp\Big\{\sqrt{\frac{\expec[ W_{\sss N}]}{\nu_{\sss N}}} |zr| + \frac{r^2}{N^{1/2}} 
	- N d\Big(\sqrt{\frac{\expec[ W_{\sss N}]}{\nu_{\sss N}}}\frac{z}{N^{1/4}}\Big)\Big\},
	\ee
which we use as the dominating function. Hence, we need to prove that the integral of this function over $z\in\mathbb{R}$ is uniformly bounded. We split the integral as
	\begin{align}
	\int_{-\infty}^{\infty} \exp\Big\{&\sqrt{\frac{\expec[ W_{\sss N}]}{\nu_{\sss N}}} |zr| + \frac{r^2}{N^{1/2}} 
	- N d	\Big(\sqrt{\frac{\expec[ W_{\sss N}]}{\nu_{\sss N}}}\frac{z}{N^{1/4}}\Big)\Big\} \dint z  \nn\\
	&=\int_{\left|\sqrt{\frac{\expec[ W_{\sss N}]}{\nu_{\sss N}}}\frac{z}{N^{1/4}}\right|\leq A} 
	\exp\Big\{\sqrt{\frac{\expec[ W_{\sss N}]}{\nu_{\sss N}}} |zr| + \frac{r^2}{N^{1/2}} 
	- N d\Big(\sqrt{\frac{\expec[ W_{\sss N}]}{\nu_{\sss N}}}\frac{z}{N^{1/4}}\Big)\Big\} \dint z \nn\\
	&\qquad+ \int_{\left|\sqrt{\frac{\expec[ W_{\sss N}]}{\nu_{\sss N}}}\frac{z}{N^{1/4}}\right|>A} 
	\exp\Big\{\sqrt{\frac{\expec[ W_{\sss N}]}{\nu_{\sss N}}} |zr| + \frac{r^2}{N^{1/2}} 
	- N d\Big(\sqrt{\frac{\expec[ W_{\sss N}]}{\nu_{\sss N}}}\frac{z}{N^{1/4}}\Big)\Big\} \dint z \;.
	\end{align}
The first integral equals
	\be
	\int_{\left|\sqrt{\frac{\expec[ W_{\sss N}]}{\nu_{\sss N}}}\frac{z}{N^{1/4}}\right|\leq A} 
	\exp\Big\{\sqrt{\frac{\expec[ W_{\sss N}]}{\nu_{\sss N}}} |zr| + \frac{r^2}{N^{1/2}} 
	- \varepsilon \frac{\expec[ W_{\sss N}]^4}{\mathbb{E}[W_{\sss N}^2]^2}z^4\Big)\Big\} \dint z,
	\ee
which clearly is uniformly bounded. The second integral equals
	\begin{align}
	\int_{\left|\sqrt{\frac{\expec[ W_{\sss N}]}{\nu_{\sss N}}}\frac{z}{N^{1/4}}\right|>A} 
	& \exp\Big\{\sqrt{\frac{\expec[ W_{\sss N}]}{\nu_{\sss N}}} |zr| + \frac{r^2}{N^{1/2}} 
	- \varepsilon \frac{\expec[ W_{\sss N}]}{\nu_{\sss N}}z^2 \sqrt{N}\Big)\Big\} \dint z \\
	&=\frac{1}{N^{1/4}}\int_{\left|\sqrt{\frac{\expec[ W_{\sss N}]}{\nu_{\sss N}}}\frac{y}{N^{1/2}}\right|>A} 
	 \exp\Big\{\sqrt{\frac{\expec[ W_{\sss N}]}{\nu_{\sss N}}} \frac{|yr|}{N^{1/4}} + \frac{r^2}{N^{1/2}} 
	 - \varepsilon \frac{\expec[ W_{\sss N}]}{\nu_{\sss N}}y^2 \Big)\Big\} \dint y,\nn
	\end{align}
where we have substituted $y=z N^{1/4}$. This converges to zero for $N\rightarrow \infty$, because the integral is uniformly bounded.

Together with the pointwise convergence proved in Lemma~\ref{pluto}, this proves Lemma \ref{lem-integralconvE4}.
\end{proof}


\subsection{Convergence for $\tau\in(3,5)$}\label{sec-convergence-tau35}
We next analyze $G_{\sss N}(z;r)$ for $\tau\in(3,5)$, assuming Condition~\ref{cond-det}. 

\begin{lemma}[Asymptotics of $G_{\sss N}$ for $\tau\in(3,5)$]\label{lem-plutotau35}
Assume that Condition~\ref{cond-det}(ii) holds. For $\beta=\beta_{c,N}$, $B=0$ and $\tau\in(3,5)$,
	\be
	\lim_{N\rightarrow\infty} N G_{\sss N}(z/N^{1/(\tau-1)};r) 
	= -z r \sqrt{\frac{\expec[W]}{\nu}}+f\left(\sqrt{\frac{\expec[W]}{\nu}}z\right),
	\ee
where $f(z)$ is defined in~\eqref{eq-limitingdensity}.
\end{lemma}

\begin{proof}
Define the function
	\be\label{eq-defgwz}
	g(w,z) = \frac{1}{2}\left( \alpha_{\sss N}(\beta_{c,N})w z\right)^2-\log \cosh \left( \alpha_{\sss N}(\beta_{c,N})w z\right),
	\ee
so that we can rewrite, in a similar way as in~\eqref{eq-rewriteGN},
	\begin{align}\label{eq-rewriteGNtau35}
	N G_{\sss N}\left(z/N^{\frac{1}{\tau-1}};r \right) 
	&= N \expec[g(W_{\sss N}, z/N^{\frac{1}{\tau-1}})] \\
	&\quad- N \expec\left[ \log \cosh \Big(\frac{1}{\sqrt{\mathbb{E}[W_{\sss N}^2]}}W_{\sss N}\frac{z}{N^{\frac{1}{\tau-1}}}
	+ \frac{r}{N^{\lambda}}\Big) 
	-\log \cosh \Big(\frac{1}{\sqrt{\mathbb{E}[W_{\sss N}^2]}}W_{\sss N}\frac{z}{N^{\frac{1}{\tau-1}}}\Big)  \right].\nn
	\end{align}

By the definition of $W_{\sss N}$, we can rewrite
	\be
	\expec[g(W_{\sss N},z/N^{1/(\tau-1)})] = \frac1N\sum_{i=1}^N g(w_i,z/N^{1/(\tau-1)}).
	\ee
With the deterministic choice of the weights as in~\eqref{eq-precisepowerlaw},
	\begin{align}
	g(w_i,z/N^{1/(\tau-1)})&= \frac{1}{2}\left( \alpha_{\sss N}(\beta_{c,N})\frac{w_i z}{N^{1/(\tau-1)}}\right)^2 
	-\log \cosh \left( \alpha_{\sss N}(\beta_{c,N})\frac{w_i z}{N^{1/(\tau-1)}} \right)\nn\\
	&=\frac{1}{2}\left( \frac{1}{\sqrt{\expec[W_{\sss N}^2]}} \frac{c_w z}{i^{1/(\tau-1)}}  \right)^2 
	-\log \cosh \left( \frac{1}{\sqrt{\expec[W_{\sss N}^2]}}\frac{c_w z}{i^{1/(\tau-1)}} \right).
	\end{align}
From this it clearly follows that, for all $i\geq1$,
	\be
	\lim_{N\rightarrow\infty} g(w_i,z/N^{1/(\tau-1)}) 
	= \frac{1}{2}\left( \frac{1}{\sqrt{\expec[W^2]}} \frac{c_w z}{i^{1/(\tau-1)}}\right)^2 
	-\log \cosh \left( \frac{1}{\sqrt{\expec[W^2]}}\frac{c_w z}{i^{1/(\tau-1)}} \right).
	\ee
It remains to show that also the sum converges, which we do using dominated convergence. For this, we use a Taylor expansion of $\log\cosh(x)$ about $x=0$ up to the fourth order
	\be\label{eq-lbtaylorlogcosh}
	\log \cosh(x) = \frac{x^2}{2} + \left(-2+2\tanh^2(\xi)+6\tanh^2(\xi)(1-\tanh^2(\xi))\right)\frac{x^4}{4!} 
	\geq \frac{x^2}{2}-\frac{x^4}{12},
	\ee
for some $\xi\in(0,x)$. Hence,
	\be
	g(w_i,z/N^{1/(\tau-1)}) 
	\leq \frac{1}{12} \left( \frac{1}{\sqrt{\expec[W_{\sss N}^2]}}\frac{c_w z}{i^{1/(\tau-1)}} \right)^4
	= \frac{1}{12}  \frac{1}{\expec[W^2]^2+o(1)} \frac{(c_w z)^4}{i^{4/(\tau-1)}} .
	\ee
Since $\tau\in(3,5)$, it holds that $4/(\tau-1)>1$, so that
	\be
	\lim_{N\rightarrow\infty} \sum_{i=1}^N \frac{1}{12} \frac{1}{\expec[W^2]^2+o(1)} \frac{(c_w z)^4}{i^{4/(\tau-1)}} < \infty.
	\ee
We conclude that
	\begin{align}
	\lim_{N\rightarrow\infty} \sum_{i=1}^N g(w_i,z/N^{1/(\tau-1)}) 
	&= \sum_{i=1}^\infty \left(\frac{1}{2}\left( \frac{1}{\sqrt{\expec[W^2]}} \frac{c_w z}{i^{1/(\tau-1)}}\right)^2 
	-\log \cosh \left( \frac{1}{\sqrt{\expec[W^2]}}\frac{c_w z}{i^{1/(\tau-1)}} \right)\right)\nn\\
	&= f\left(\sqrt{\frac{\expec[W]}{\nu}}z\right),
	\end{align}
where in the last equality we have used that $\expec[W]= c_w \frac{\tau-1}{\tau-2}$.
This is in turn a consequence of the following explicit computation giving an upper and lower bound on $\expec[W_{\sss N}]$ matching in the limit $N\to\infty$. An upper bound on the first moment is given by
	\begin{align}
	\expec[W_{\sss N}] &= \frac{1}{N}\sum_{i=1}^N c_w\left(\frac{N}{i}\right)^{1/(\tau-1)} 
	=c_w N^{-\frac{\tau-2}{\tau-1}}\sum_{i=1}^N c_w i^{-1/(\tau-1)} 
	\leq c_w N^{-\frac{\tau-2}{\tau-1}} \left(1+\int_{1}^N i^{-1/(\tau-1)} \dint i\right)\nn\\
	& = c_w \frac{\tau-1}{\tau-2} - c_w\frac{1}{\tau-2} N^{-\frac{\tau-2}{\tau-1}},
	\end{align}
and a lower bound by
	\be
	\expec[W_{\sss N}] \geq c_w N^{-\frac{\tau-2}{\tau-1}} \int_{1}^N i^{-1/(\tau-1)} \dint i 
	= c_w \frac{\tau-1}{\tau-2} - c_w\frac{\tau-1}{\tau-2} N^{-\frac{\tau-2}{\tau-1}}.
	\ee
From this it indeed follows that
	\be
	\expec[W]=\lim_{N\rightarrow\infty} \expec[W_{\sss N}] = c_w \frac{\tau-1}{\tau-2}.
	\ee

To analyze the second term in~\eqref{eq-rewriteGNtau35}, we can use the Taylor expansions
	\begin{align}
	\log \cosh(a+x) &= \log \cosh(a) + \tanh(a) x + (1-\tanh^2(\xi)) x^2 \nn\\
	&= \log \cosh(a) + (a-\tanh\zeta(1-\tanh^2\zeta)a^2) x + (1-\tanh^2(\xi)) x^2,
	\end{align}
for some $\xi\in(a,a+x)$ and $\zeta\in(0,a)$. This gives
	\begin{align}
	N \expec\Big[ \log \cosh & \Big(\frac{1}{\sqrt{\mathbb{E}[W_{\sss N}^2]}}W_{\sss N}\frac{z}{N^{\frac{1}{\tau-1}}}
	+ \frac{r}{N^{\lambda}}\Big) 
	-\log \cosh \Big(\frac{1}{\sqrt{\mathbb{E}[W_{\sss N}^2]}}W_{\sss N}\frac{z}{N^{\frac{1}{\tau-1}}}\Big)  \Big] \nn\\
	&=N \sqrt{\frac{\expec[W_{\sss N}]}{\nu_{\sss N}}}\frac{z}{N^{\frac{1}{\tau-1}}}\frac{r}{N^{\lambda}} 
	- N\expec\Big[\tanh\zeta(1-\tanh^2\zeta) W_{\sss N}^2 \Big]
	\frac{1}{\mathbb{E}[W_{\sss N}^2]}\frac{z^2}{N^{\frac{2}{\tau-1}}}\frac{r}{N^{\lambda}} \nn\\
	&\qquad\ \qquad +N \expec\Big[ (1-\tanh^2(\xi)) \Big]\frac{r^2}{N^{2\lambda}} \nn\\
	&= \sqrt{\frac{\expec[W_{\sss N}]}{\nu_{\sss N}}} zr + o(1),
	\end{align}
where the last equality follows from $\lambda=\frac{\tau-2}{\tau-1}$ and $\tau\in(3,5)$.
\end{proof}
\medskip

Again it follows that also the integral converges:
\begin{lemma}[Convergence of the integral for $\tau\in(3,5)$]
\label{lem-integralconvtau35}
For $\beta=\beta_{c,N}$, $B=0$ and $\tau\in(3,5)$,
	\be
	\lim_{N\rightarrow\infty}\int_{-\infty}^\infty \e^{- N G_{\sss N}(z/N^{1/(\tau-1)};r)} \dint z 
	= \int_{-\infty}^\infty \e^{z r \sqrt{\frac{\expec[W]}{\nu}}-f\left(\sqrt{\frac{\expec[W]}{\nu}}z\right)} \dint z.
	\ee
\end{lemma}

\begin{proof}
We again start from the rewrite of $G_{\sss N}$ in~\eqref{eq-rewriteGN}. As before,
	\be
	N \expec[g(W_{\sss N}, z/N^{\frac{1}{\tau-1}})] 
	= \sum_{i=1}^N \Big[\frac{1}{2}\Big( \frac{1}{\sqrt{\expec[W_{\sss N}^2]}} \frac{c_w z}{i^{1/(\tau-1)}}\Big)^2 
	-\log \cosh \Big( \frac{1}{\sqrt{\expec[W_{\sss N}^2]}}\frac{c_w z}{i^{1/(\tau-1)}}\Big)\Big],
	\ee
where it is easy to see that the summands are positive and decreasing in $i$. Hence,
	\be
	N \expec[g(W_{\sss N}, z/N^{\frac{1}{\tau-1}})] 
	\geq  \int_{1}^N \Big[\frac{1}{2}\Big( \frac{1}{\sqrt{\expec[W_{\sss N}^2]}} \frac{c_w z}{y^{1/(\tau-1)}}\Big)^2 
	-\log \cosh \Big( \frac{1}{\sqrt{\expec[W_{\sss N}^2]}}\frac{c_w z}{y^{1/(\tau-1)}}\Big)\Big] \dint y.
	\ee
We want to use~\eqref{eq-taylorh24}, and hence split the integral in the region where $|\frac{1}{\sqrt{\expec[W_{\sss N}^2]}} \frac{c_w z}{y^{1/(\tau-1)}}|$ is bigger or smaller than $A$. This gives
	\begin{align}
	N \expec[g(W_{\sss N}, z/N^{\frac{1}{\tau-1}})] 
	&\geq \varepsilon\frac{c_w^2 z^2}{\expec[W_{\sss N}^2]} 
	\int_{1}^{\Big(\frac{A\sqrt{\expec[W_{\sss N}^2]}}{c_w |z|}\Big)^{\tau-1}} \frac{1}{y^{\frac{2}{\tau-1}}} \dint y 
	+ \varepsilon \frac{c_w^4 z^4}{\expec[W_{\sss N}^2]} 
	\int_{\Big(\frac{A\sqrt{\expec[W_{\sss N}^2]}}{c_w |z|}\Big)^{\tau-1}}^N \frac{1}{y^{\frac{4}{\tau-1}}} \dint y \nn\\
	&=\varepsilon\frac{c_w^2 z^2}{\expec[W_{\sss N}^2]}
	\frac{\tau-1}{\tau-3} \left(\Big(\frac{A\sqrt{\expec[W_{\sss N}^2]}}{c_w |z|}\Big)^{\tau-3}-1\right) \nn\\
	&\qquad\ \qquad - \varepsilon \frac{c_w^4 z^4}{\expec[W_{\sss N}^2]} \frac{\tau-1}{5-\tau}
	\left(N^{-\frac{5-\tau}{\tau-1}}-\Big(\frac{A\sqrt{\expec[W_{\sss N}^2]}}{c_w |z|}\Big)^{-(5-\tau)}\right) \nn\\
	&=k_1 |z|^{-(\tau+1)}-k_2 z^2-o(1) z^4+k_3|z|^{9-\tau},
	\end{align}
for the proper constants $k_1,k_2,k_3>0$. Since $9-\tau>4$,
	\be
	\int_{-\infty}^{\infty} \e^{-k_1 |z|^{-(\tau+1)}+k_2 z^2+o(1) z^4-k_3|z|^{9-\tau}} \dint z <\infty.
	\ee
Together with the pointwise convergence in the previous lemma, this proves this lemma for $r=0$. For $r\neq0$, the proof can be adapted as for the case $\expec[W^4]<\infty$.
\end{proof}

We next analyze the large $x$ behavior of $f(x)$ arising in the density of the limiting random variable:
\begin{lemma}[Asymptotics of $f$ for $\tau\in(3,5)$]
\label{lem-largeztau35}
For $\tau\in(3,5)$,
	\be
	\lim_{x\rightarrow\infty}\frac{f(x)}{x^{\tau-1}} 
	= \left(\frac{\tau-2}{\tau-1}\right)^{\tau-1} \int_0^\infty 
	\left(\frac{1}{2y^{2/(\tau-1)}}- \log\cosh\frac{1}{y^{1/(\tau-1)}}\right) \dint y <\infty.
	\ee
\end{lemma}

\begin{proof}
We first prove that the integral is finite. For this, define
	\be
	h(y) = \frac12 y^2 -\log\cosh y,
	\ee
so that $h(y)\geq 0$.
Then,
	\be
	\int_0^\infty \left(\frac{1}{2y^{2/(\tau-1)}}
	- \log\cosh\frac{1}{y^{1/(\tau-1)}}\right) \dint y = \int_0^\infty h\left(\frac{1}{y^{1/(\tau-1)}}\right) \dint y.
	\ee
Since $\log\cosh y\ge0$, we have  $h(y) \leq \frac12 y^2$, and hence
	\be
	h\left(\frac{1}{y^{1/(\tau-1)}}\right) \leq \frac{1}{2y^{2/(\tau-1)}}.
	\ee
This is integrable for $y\rightarrow 0$, because $2/(\tau-1)<1$ for $\tau\in(3,5)$.

Using~\eqref{eq-lbtaylorlogcosh}, for $y$ large,
	\be
	h\left(\frac{1}{y^{1/(\tau-1)}}\right) \leq \frac{1}{12} \frac{1}{y^{4/(\tau-1)}}.
	\ee
This is integrable for $y\rightarrow\infty$, because $4/(\tau-1) >1$ for $\tau\in(3,5)$.

To prove that $f(x)/x^{\tau-1}$ converges to the integral as $x\rightarrow\infty$ we rewrite, with $a=(\tau-2)/(\tau-1)$,
	\begin{align}
	\frac{f(x)}{x^{\tau-1}} &= \frac{1}{x^{\tau-1}}\sum_{i =1}^\infty h\left(a\frac{x}{i^{1/(\tau-1)}}\right)
	=a^{\tau-1} \frac{1}{\left(ax\right)^{\tau-1}}
	\sum_{i =1}^\infty h\Big(\Big(\frac{1}{i/\left(ax\right)^{\tau-1}}\Big)^{1/(\tau-1)}\Big)\\
	&=a^{\tau-1} \int_0^\infty h\left(\frac{1}{y^{1/(\tau-1)}}\right) \dint y \left(1+o(1)\right).\nn
	\end{align}
\vskip-1.1cm
\end{proof}
\vskip0.5cm


\subsection{Proof of Theorem~\ref{thm-nonclassicalclt}}
\label{final-proof}

We can now prove Theorem~\ref{thm-nonclassicalclt} for the measure $\widetilde{P}_{\sss N}$:
\begin{proof}[Proof of Theorem~\ref{thm-nonclassicalclt} for the measure $\widetilde{P}_{\sss N}$]
We can do a change of variables so that
	\be\label{eq-changevariablendelta1}
	\int_{-\infty}^{\infty} \e^{-N G_{\sss N}(z;r)}\dint z 
	= N^{1/(\boldsymbol{\delta}+1)} \int_{-\infty}^{\infty} \e^{-N G_{\sss N}(z/N^{1/(\boldsymbol{\delta}+1)};r)}\dint z.
	\ee
Hence, using Lemma \ref{pippo}
	\be
	\widetilde{P}_{\sss N}\Big( \exp\Big(r \frac{S_{\sss N}}{N^{\boldsymbol{\delta}/(\boldsymbol{\delta}+1)}}\Big)\Big) 
	= \frac{\int_{-\infty}^{\infty} \e^{-N G_{\sss N}(z/N^{1/(\boldsymbol{\delta}+1)};r)}\dint z}
	{\int_{-\infty}^{\infty} \e^{-N G_{\sss N}(z/N^{1/(\boldsymbol{\delta}+1)};0)}\dint z} \;.
	\ee
It follows from Lemma~\ref{lem-integralconvE4} for $\mathbb{E}[W^4] < \infty$
and from Lemma~\ref{lem-integralconvtau35} for $\tau\in (3,5)$ that
	\be\label{eq-conv-apprx-momentgen}
	\lim_{N\rightarrow\infty}\widetilde{P}_{\sss N}
	\Big( \exp\Big(r \frac{S_{\sss N}}{N^{\boldsymbol{\delta}/(\boldsymbol{\delta}+1)}}\Big)\Big) 
	= \frac{\int_{-\infty}^{\infty} \e^{z r \sqrt{\frac{\expec[W]}{\nu}}-f\left(\sqrt{\frac{\expec[W]}{\nu}}z\right)}\dint z }
	{\int_{-\infty}^{\infty} \e^{-f\left(\sqrt{\frac{\expec[W]}{\nu}}z\right)}\dint z }
	= \frac{\int_{-\infty}^{\infty} \e^{x r-f(x)}\dint x }{\int_{-\infty}^{\infty} \e^{-f(x)}\dint x },
	\ee
where we made the change of variables $x=\sqrt{\frac{\expec[W]}{\nu}}z$ in both integrals to obtain the last equality.

As mentioned, this is sufficient to prove the convergence in distribution of $\frac{S_{\sss N}}{N^{\boldsymbol{\delta}/(\boldsymbol{\delta}+1)}}$ to the random variable $X$ ~(see \cite[Theorem~A.8.7(a)]{El}).

For the case $\expec[W^4]<\infty$,
	\be
	\lim_{x\rightarrow\infty}\frac{f(x)}{x^{1+\boldsymbol{\delta}}} 
	= \lim_{x\rightarrow\infty} \frac{\frac{1}{12}\frac{\expec[W^4]}{\expec[W]^4}x^4}{x^4}
	=\frac{1}{12}\frac{\expec[W^4]}{\expec[W]^4}.
	\ee
For $\tau\in(3,5)$, the proof that $\lim_{x\rightarrow\infty}\frac{f(x)}{x^{1+\boldsymbol{\delta}}} = C$ is given in Lemma~\ref{lem-largeztau35}.
\end{proof}

It remains to show that the statement of Theorem~\ref{thm-nonclassicalclt} also holds for the measure $P_{\sss N}$. This follows from the following lemma:

\begin{lemma}\label{lem-fromtildetoreal}
For $\expec[W^4]<\infty$ and $\tau\in(3,5)$,
	\be
	\lim_{N_\rightarrow\infty} P_{\sss N}\Big( \exp\Big(r \frac{S_{\sss N}}{N^{\lambda}}\Big)\Big)
	-\widetilde{P}_{\sss N}\Big( \exp\Big(r \frac{S_{\sss N}}{N^{\lambda}}\Big)\Big)=0.
	\ee
\end{lemma}

\begin{proof}
As shown in~\cite{GGvdHP},
	\be\label{eq-measurePN}
	P_{\sss N}\big( \col{g}(\sigma)\big) 
	= \frac{\sum_{\sigma \in \Omega_{\sss N}} \col{g}(\sigma) \e^{\frac12 \sum_{i,j\in[N]}\col{J_{ij}}\sigma_i\sigma_j}}
	{\sum_{\sigma \in \Omega_{\sss N}} \e^{\frac12 \sum_{i,j\in[N]}\col{J_{ij}}\sigma_i\sigma_j}},
	\ee
where
	\be
	\label{couplings-expansion}
	\col{J_{ij}} = \frac12 \log\left(\frac{e^{\beta}p_{ij}+(1-p_{ij})}{e^{-\beta}p_{ij}+(1-p_{ij})}\right)
	=p_{ij}\sinh\beta-p_{ij}^2\sinh\beta(\cosh\beta-1)+O(p_{ij}^3),
	\ee
where we have used the Taylor expansion of $\log(1+x)$ about $x=0$ in the last equality. Hence, using~\eqref{eq-prob-edge}, 
	\begin{align}
	\e^{\frac12 \sum_{i,j\in[N]}\col{J_{ij}}\sigma_i\sigma_j} 
	&= \e^{\frac12 \sum_{i,j\in[N]}\left(\left(\frac{w_i w_j}{\ell_{\sss N}+w_i w_j}
	-\frac{w_i w_j}{\ell_{\sss N}}\right)\sinh\beta-p_{ij}^2\sinh\beta(\cosh\beta-1)
	+O(p_{ij}^3)\right)\sigma_i\sigma_j}e^{\frac12 \sinh \beta \sum_{i,j\in[N]}\frac{w_i w_j}{\ell_{\sss N}}\sigma_i\sigma_j} \nn\\
	&=: \e^{E_{\sss N}(\sigma)}e^{\frac12 \sinh \beta \sum_{i,j\in[N]}\frac{w_i w_j}{\ell_{\sss N}}\sigma_i\sigma_j}.
	\end{align}
Hence, we can rewrite~\eqref{eq-measurePN} as
	\be
	P_{\sss N}\big( \col{g}(\sigma)\big) 
	=  \frac{\sum_{\sigma \in \Omega_{\sss N}} \col{g}(\sigma)  \e^{E_{\sss N}(\sigma)}
	\e^{\frac12 \sinh \beta \sum_{i,j\in[N]}\frac{w_i w_j}{\ell_{\sss N}}\sigma_i\sigma_j}}
	{\sum_{\sigma \in \Omega_{\sss N}}  \e^{E_{\sss N}(\sigma)}
	\e^{\frac12 \sinh \beta \sum_{i,j\in[N]}\frac{w_i w_j}{\ell_{\sss N}}\sigma_i\sigma_j}}
	= \frac{\widetilde{P}_{\sss N}\big( \col{g}(\sigma)\e^{E_{\sss N}(\sigma)}\big)}
	{\widetilde{P}_{\sss N}\big(\e^{E_{\sss N}(\sigma)}\big)}.
	\ee
Combining this with the Cauchy-Schwarz inequality gives
	\begin{align}
	\left| P_{\sss N}\Big( \exp\Big(r \frac{S_{\sss N}}{N^{\lambda}}\Big)\Big)-\widetilde{P}_{\sss N}\Big( \exp\Big(r \frac{S_{\sss N}}{N^{\lambda}}\Big)\Big)\right| 
	& = \frac{\left|\widetilde{P}_{\sss N}\Big( \exp\Big(r \frac{S_{\sss N}}{N^{\lambda}}\Big) 
	\Big(\e^{E_{\sss N}(\sigma)}- \widetilde{P}_{\sss N}\big(\e^{E_{\sss N}(\sigma)}\big)\Big)\Big)\right|}
	{\widetilde{P}_{\sss N}\big(\e^{E_{\sss N}(\sigma)}\big)} \nn\\
	&\leq \frac{\sqrt{\widetilde{P}_{\sss N}\Big( \exp\Big(2r \frac{S_{\sss N}}{N^{\lambda}}\Big)\Big)} \sqrt{\widetilde{P}_{\sss N}
	\bigg(\Big(\e^{E_{\sss N}(\sigma)}- \widetilde{P}_{\sss N}\big(\e^{E_{\sss N}(\sigma)}\big)\Big)^2\bigg)} }
	{\widetilde{P}_{\sss N}\big(\e^{E_{\sss N}(\sigma)}\big)}\nn\\
	&= \frac{\sqrt{\widetilde{P}_{\sss N}\Big( \exp\Big(2r \frac{S_{\sss N}}{N^{\lambda}}\Big)\Big)} 
	\sqrt{\widetilde{P}_{\sss N}\Big(\e^{2E_{\sss N}(\sigma)}\Big)
	- \widetilde{P}_{\sss N}\big(\e^{E_{\sss N}(\sigma)}\big)^2} }{\widetilde{P}_{\sss N}\big(\e^{E_{\sss N}(\sigma)}\big)}\;.
	\end{align}
From~\eqref{eq-conv-apprx-momentgen}, it follows that the first square root converges as $N\to\infty$. We next analyze $E_{\sss N}(\sigma)$ and show that $E_{\sss N}(\sigma)\to 0$ in probability w.r.t.\ $\widetilde{P}_{\sss N}$. We also show that $E_{\sss N}(\sigma)$ is uniformly bounded from above, so that the lemma follows by dominated convergence.

We first analyze the contribution of the $O(p_{ij}^3)$ terms in $E_{\sss N}(\sigma)$. Note that
	\be
	\sum_{i,j\in[N]} p_{ij}^3 = \sum_{i,j\in[N]} \left(\frac{w_iw_j}{\ell_{\sss N}+w_iw_j}\right)^3
	 \leq \sum_{i,j\in[N]} \left(\frac{w_iw_j}{\ell_{\sss N}}\right)^3 = \frac{1}{\ell_{\sss N}^3} \bigg(\sum_{i\in[N]} w_i^3\bigg)^2.
	\ee
For $\expec[W_{\sss N}^2]\to\expec[W^2]<\infty$ it holds that $\max_i w_i=o(\sqrt{N})$. Hence,
	\be
	\bigg(\sum_{i\in[N]} w_i^3\bigg)^2 \leq  (\max_i w_i)^2 \bigg(\sum_{i\in[N]} w_i^2\bigg)^2 
	= o(N^3) \expec[W_{\sss N}^2]^2 = o(\ell_{\sss N}^3),
	\ee
because $\ell_{\sss N} = O(N)$. Hence,
	\begin{align}
	E_{\sss N}(\sigma) 
	&= \frac12 \sum_{i,j\in[N]}\left(\left(\frac{w_i w_j}{\ell_{\sss N}+w_i w_j}-\frac{w_i w_j}{\ell_{\sss N}}\right)\sinh\beta
	-p_{ij}^2\sinh\beta(\cosh\beta-1)+\mathcal{O}(p_{ij}^3)\right)\sigma_i\sigma_j \nn\\
	&= -\frac12 \sum_{i,j\in[N]}\bigg(\frac{w_i^2 w_j^2}{\ell_{\sss N}(\ell_{\sss N}+w_i w_j)}
	\sinh\beta+\left(\frac{w_i w_j}{\ell_{\sss N}+w_i w_j}\right)^2\sinh\beta(\cosh\beta-1)\bigg)\sigma_i\sigma_j +o(1) \nn\\
	&= -\frac12\sinh\beta\cosh\beta \sum_{i,j\in[N]}\frac{w_i^2 w_j^2}{\ell_{\sss N}^2}\sigma_i\sigma_j +o(1) \nn\\
	&= - \frac12\sinh\beta\cosh\beta \bigg(\sum_{i\in[N]}\frac{w_i^2}{\ell_{\sss N}} \sigma_i \bigg)^2 +o(1),
	\end{align}
where the third equality can be proved as in the analysis of $p_{ij}^3$. Hence, $E_{\sss N}(\sigma)$ is indeed uniformly bounded from above, so that $\e^{E_{\sss N}(\sigma)}$ is uniformly bounded. 

It remains to prove that $E_{\sss N}(\sigma)\to 0$ in probability w.r.t.\ $\widetilde{P}_{\sss N}$. We define $Y_{\sss N} = \sum_{i\in[N]}\frac{w_i^2}{\ell_{\sss N}} \sigma_i$, so that
	\be
	E_{\sss N} = -\frac12 \sinh\beta\cosh\beta \, Y_{\sss N}^2 + o(1).
	\ee
We analyze the moment generating function of $Y_{\sss N}$ the same way as $S_{\sss N}/N^\lambda$. That is, we use the Hubbard-Stratonovich identity to rewrite
	\begin{align}
	\widetilde{P}_{\sss N} \left(\e^{r Y_{\sss N}}\right) 
	&= \frac{\sum_{\sigma\in\Omega_{\sss N}} \e^{r Y_{\sss N}} \e^{\frac12\frac{\sinh\beta}{\ell_{\sss N}}
	(\sum_{i\in[N]}w_i\sigma_i)^2}}{\sum_{\sigma\in\Omega_{\sss N}} 
	\e^{\frac12\frac{\sinh\beta}{\ell_{\sss N}}(\sum_{i\in[N]}w_i\sigma_i)^2}}\\
	&= \frac{\sum_{\sigma\in\Omega_{\sss N}} \expec\left[\e^{r\sum_{i\in[N]}\frac{w_i^2}
	{N\expec[W_{\sss N}]} \sigma_i + \sqrt{\frac{\sinh\beta}{N\expec[W_{\sss N}]}}
	\sum_{i\in[N]}w_i\sigma_i Z}\right]}{\sum_{\sigma\in\Omega_{\sss N}} 
	\expec\left[e^{\sqrt{\frac{\sinh\beta}{N\expec[W_{\sss N}]}}\sum_{i\in[N]}w_i\sigma_i Z }\right]}\nn\\
	& = \frac{\expec\bigg[\e^{N\expec\big[\log\cosh\big(r\frac{W_{\sss N}^2}{N\expec[W_{\sss N}]}+\sqrt{\frac{\sinh\beta}{N\expec[W_{\sss N}]}}W_{\sss N} Z\big)\,\big|\,Z\big]}\bigg]}{\expec\bigg[\e^{N\expec\big[\log\cosh\big(\sqrt{\frac{\sinh\beta}{N\expec[W_{\sss N}]}}W_{\sss N} Z\big)\,\big|\,Z\big]}\bigg]}\nn\\
	&=\frac{\int_{-\infty}^{\infty} \e^{-z^2/2+N\expec\big[\log\cosh\big(r\frac{W_{\sss N}^2}{N\expec[W_{\sss N}]}+\sqrt{\frac{\sinh\beta}{N\expec[W_{\sss N}]}}W_{\sss N} z\big)\big]} \dint z}{\int_{-\infty}^{\infty} \e^{-z^2/2+N\expec\big[\log\cosh\big(\sqrt{\frac{\sinh\beta}{N\expec[W_{\sss N}]}}W_{\sss N} z\big)\big]} \dint z}.\nn
	\end{align}
We do a change of variables replacing $z/\sqrt{N}$ by $z$, so that
	\begin{align}\label{eq-momgenYN}
	\widetilde{P}_{\sss N} \left(\e^{r Y_{\sss N}}\right)
	&=\frac{\int_{-\infty}^{\infty} \e^{-Nz^2/2+N\expec\big[\log\cosh\big(r\frac{W_{\sss N}^2}{N\expec[W_{\sss N}]}+\sqrt{\frac{\sinh\beta}{\expec[W_{\sss N}]}}W_{\sss N} z\big)\big]} \dint z}{\int_{-\infty}^{\infty} \e^{-Nz^2/2+N\expec\big[\log\cosh\big(\sqrt{\frac{\sinh\beta}{\expec[W_{\sss N}]}}W_{\sss N} z\big)\big]} \dint z} \\
&= \frac{\int_{-\infty}^{\infty} \e^{-NG_{\sss N}(z;0)+N\expec\big[\log\cosh\big(r\frac{W_{\sss N}^2}{N\expec[W_{\sss N}]}+\sqrt{\frac{\sinh\beta}{\expec[W_{\sss N}]}}W_{\sss N} z\big)-\log\cosh\big(\sqrt{\frac{\sinh\beta}{\expec[W_{\sss N}]}}W_{\sss N} z\big)\big]} \dint z}{\int_{-\infty}^{\infty} \e^{-NG_{\sss N}(z;0)} \dint z}\nn\\
&= \frac{\int_{-\infty}^{\infty} \e^{-NG_{\sss N}(z/N^{1/(\boldsymbol{\delta}+1)};0)+N\expec\big[\log\cosh\big(r\frac{W_{\sss N}^2}{N\expec[W_{\sss N}]}+\sqrt{\frac{\sinh\beta}{\expec[W_{\sss N}]}}W_{\sss N} \frac{z}{N^{1/(\boldsymbol{\delta}+1)}}\big)-\log\cosh\big(\sqrt{\frac{\sinh\beta}{\expec[W_{\sss N}]}}W_{\sss N} \frac{z}{N^{1/(\boldsymbol{\delta}+1)}}\big)\big]} \dint z}{\int_{-\infty}^{\infty} \e^{-NG_{\sss N}(z/N^{1/(\boldsymbol{\delta}+1)};0)} \dint z},\nn
\end{align}
where we did another change of variable in the last equality.

In Lemmas~\ref{pluto} and~\ref{lem-plutotau35}, we proved that $NG_{\sss N}(z/N^{1/(\boldsymbol{\delta}+1)};0)$ converges for $\beta=\beta_c$. We Taylor expand the remaining term,
\begin{align}
N\expec\Big[&\log\cosh\Big(r\frac{W_{\sss N}^2}{N\expec[W_{\sss N}]}+\sqrt{\frac{\sinh\beta}{\expec[W_{\sss N}]}}W_{\sss N} \frac{z}{N^{1/(\boldsymbol{\delta}+1)}}\Big)-\log\cosh\Big(\sqrt{\frac{\sinh\beta}{\expec[W_{\sss N}]}}W_{\sss N} \frac{z}{N^{1/(\boldsymbol{\delta}+1)}}\Big)\Big]\nn\\
&= \expec\Big[\tanh\Big(\sqrt{\frac{\sinh\beta}{\expec[W_{\sss N}]}}W_{\sss N} \frac{z}{N^{1/(\boldsymbol{\delta}+1)}}\Big)r\frac{W_{\sss N}^2}{\expec[W_{\sss N}]}+o(1)\Big].
\end{align}
For $\expec[W_{\sss N}^3]\to\expec[W^3]<\infty$, which includes power-law distributions with $\tau>4$, we can use that $|\tanh(x)| \leq |x|$, so that
\be
\Big|\expec\Big[\tanh\Big(\sqrt{\frac{\sinh\beta}{\expec[W_{\sss N}]}}W_{\sss N} \frac{z}{N^{1/(\boldsymbol{\delta}+1)}}\Big)r\frac{W_{\sss N}^2}{\expec[W_{\sss N}]}\Big]\Big| \leq \sqrt{\frac{\sinh\beta}{\expec[W_{\sss N}]}} \frac{|zr|}{N^{1/(\boldsymbol{\delta}+1)}}\frac{\expec[W_{\sss N}^3]}{\expec[W_{\sss N}]} = o(1).
\ee
For $\tau\in(3,4]$ we use the deterministic choice of the weights as in~\eqref{eq-precisepowerlaw} and $\boldsymbol{\delta}=\tau-2$ to rewrite
\begin{align}
\Big|\expec\Big[\tanh\Big(&\sqrt{\frac{\sinh\beta}{\expec[W_{\sss N}]}}W_{\sss N} \frac{z}{N^{1/(\boldsymbol{\delta}+1)}}\Big)r\frac{W_{\sss N}^2}{\expec[W_{\sss N}]}\Big]\Big| = \Big|\frac{1}{N}\sum_{i=1}^N \tanh\Big(\sqrt{\frac{\sinh\beta}{\expec[W_{\sss N}]}}w_i \frac{z}{N^{1/(\tau-1)}}\Big)r\frac{w_i^2}{\expec[W_{\sss N}]}\Big| \nn\\
&=\Big|\frac{1}{N}\sum_{i=1}^N \tanh\Big(\sqrt{\frac{\sinh\beta}{\expec[W_{\sss N}]}} \frac{c_w z}{i^{1/(\tau-1)}}\Big)\frac{r c_w^2}{\expec[W_{\sss N}]}\left(\frac{N}{i}\right)^{2/(\tau-1)}\Big| \nn\\
&\leq \frac{|r| c_w^2}{\expec[W_{\sss N}]} N^{2/(\tau-1)-1} + N^{2/(\tau-1)-1} \sqrt{\frac{\sinh\beta}{\expec[W_{\sss N}]}} \frac{|rz| c_w^3}{\expec[W_{\sss N}]} \sum_{i=2}^N i^{-3/(\tau-1)}.
\end{align}
For $\tau>3$ the first term is $o(1)$. For $\tau\in(3,4)$,
\be
N^{2/(\tau-1)-1} \sum_{i=2}^N i^{-3/(\tau-1)} \leq N^{-\frac{\tau-3}{\tau-1}} \int_1^N i^{-3/(\tau-1)}\dint i = \frac{\tau-1}{4-\tau} \left(N^{-\frac{\tau-3}{\tau-1}}-N^{-1/(\tau-1)}\right)=o(1),
\ee
whereas for $\tau=4$
\be
N^{2/(\tau-1)-1} \sum_{i=2}^N i^{-3/(\tau-1)} \leq N^{-\frac{\tau-3}{\tau-1}} \int_1^N i^{-3/(\tau-1)}\dint i = N^{-\frac{\tau-3}{\tau-1}} \log N=o(1).
\ee
Hence, in all cases the integrands in the numerator and denominator of~\eqref{eq-momgenYN} have the same limit. In Lemmas~\ref{lem-integralconvE4} and~\ref{lem-integralconvtau35} it is proved that the integral in the denominator converges. Since
	\be
	\Big|\expec\Big[\tanh\Big(\sqrt{\frac{\sinh\beta}
	{\expec[W_{\sss N}]}}W_{\sss N} \frac{z}{N^{1/(\boldsymbol{\delta}+1)}}\Big)r\frac{W_{\sss N}^2}{\expec[W_{\sss N}]}\Big]
	\Big| \leq \frac{r\expec[W_{\sss N}^2]}{\expec[W_{\sss N}]} = O(1),
	\ee
it follows by dominated convergence that the integral in the numerator has the same limit. Hence,
	\be
	\lim_{N\to\infty} \widetilde{P}_{\sss N} \left(\e^{r Y_{\sss N}}\right) = 1,
	\ee
from which it follows that $Y_{\sss N} \to 0$ in probability w.r.t.\ $\widetilde{P}_{\sss N}$. Hence, also $-\frac12 \sinh\beta\cosh\beta \, Y_{\sss N}^2\to 0$ in probability w.r.t.\ $\widetilde{P}_{\sss N}$. Since $o(1)$ also converges to $0$ in probability, so does the sum:
\be
E_{\sss N} = -\frac12 \sinh\beta\cosh\beta \, Y_{\sss N}^2 + o(1) \longrightarrow 0 \qquad {\rm in\ probability\ w.r.t.\ } \widetilde{P}_{\sss N}.
\ee
\end{proof}
\medskip

\begin{remark}[Sharp asymptotics of the partition function]
\label{SA-PF}
It follows from the changes of variables in~\eqref{eq-substitutesqrtN} and~\eqref{eq-changevariablendelta1} that 
$$
Z_{\sss N}(\beta_c,0)=A N^{1/2+1/(\boldsymbol{\delta}+1)} 2^{N}(1+o(1)).
$$
For $\expec[W^4]<\infty$, this exponent equals $1/2+1/(\boldsymbol{\delta}+1)=3/4$, whereas for $\tau\in(3,5)$, it is $1/2+1/(\boldsymbol{\delta}+1)=(\tau+1)/(2\tau-2)$.
\co{Thus the partition function has finite-size power-law corrections (in agreement with \cite{Bov} where the classical Curie-Weiss model is considered).}
\end{remark}


\subsection{Scaling window}\label{sec-scaling-window}
Instead of looking at the inverse temperature sequence \co{$\beta_{\sss N}=\beta_{c,N}$} we can also look at 
\co{$\beta'_{\sss N}=\beta_{c,N}+b /  N^{\frac{\boldsymbol{\delta}-1}{\boldsymbol{\delta}+1}}$} for some constant $b$. 
The analysis still works and the limiting density instead becomes
	\be
	\label{lim-dens-scal-window}
	\exp\left\{\frac{b}{2}\cosh (\beta_{c})\frac{\expec[W^2]^2}{\expec[W]^3}x^2 - f(x)\right\}.
	\ee
To see why this is correct we look at the following second moment, which  shows up in the expansion of $G_{\sss N}$, see \eqref{eq-2ndmomentargG}:
	\begin{align}\label{eq-2ndmomentscalingwindow}
	\frac12\expec&\Big[\Big(\sqrt{\frac{\sinh (\beta_{c,N}+b /  N^{\frac{\boldsymbol{\delta}-1}{\boldsymbol{\delta}+1}})}	
	{\expec[W_{\sss N}]}}W_{\sss N} \frac{z}{N^{1/(\boldsymbol{\delta}+1)}} 
	+ \frac{r}{N^{\boldsymbol{\delta}/(\boldsymbol{\delta}+1)}}\Big)^2\Big] \\
 	&\qquad = \frac{z^2}{2N^{2/(\boldsymbol{\delta}+1)}} \sinh (\beta_{c,N}+b /  N^{\frac{\boldsymbol{\delta}-1}
	{\boldsymbol{\delta}+1}})\frac{\expec[W_{\sss N}^2]}{\expec[W_{\sss N}]}
	+\sqrt{\sinh (\beta_{c,N}+b /  N^{\frac{\boldsymbol{\delta}-1}{\boldsymbol{\delta}+1}})
	\expec[W_{\sss N}]}\frac{zr}{N} + o(1/N). \nn
	\end{align}
In the first term, we Taylor expand the sine hyperbolic about \col{$\beta_{c,N}$}, which gives
	\be
	\sinh (\beta_{c,N}+b /  N^{\frac{\boldsymbol{\delta}-1}{\boldsymbol{\delta}+1}}) = \sinh (\beta_{c,N}) + 	
	\cosh(\beta_{c,N}) b /  N^{\frac{\boldsymbol{\delta}-1}{\boldsymbol{\delta}+1}} 
	+ O\left(1/	N^{2\frac{\boldsymbol{\delta}-1}{\boldsymbol{\delta}+1}}\right).
	\ee
For the other term, and also for the other terms in the expansion of $G_{\sss N}$, it suffices to note that
	\be
	\sqrt{\sinh (\beta_{c,N}+b /  N^{\frac{\boldsymbol{\delta}-1}{\boldsymbol{\delta}+1}})} 
	= \sqrt{\sinh (\beta_{c,N})} +  O\left(1/N^{\frac{\boldsymbol{\delta}-1}{\boldsymbol{\delta}+1}}\right).
	\ee
\col{	
Hence,~\eqref{eq-2ndmomentscalingwindow} equals
	\begin{align}
	& \frac{z^2}{2N^{2/(\boldsymbol{\delta}+1)}} 
	+ \frac{z^2}{2N^{2/(\boldsymbol{\delta}+1)}} \cosh (\beta_{c,N}) 
	\frac{b}{N^{\frac{\boldsymbol{\delta}-1}{\boldsymbol{\delta}+1}}} 
	\frac{\expec[W_{\sss N}^2]}{\expec[W_{\sss N}]}
	+\frac{zr}{N}\frac{\expec[W_{\sss N}]}{\sqrt{\expec[W_{\sss N}^2]}} + o(1/N)\nn\\
	&= \frac{z^2}{2N^{2/(\boldsymbol{\delta}+1)}} 
	+\frac{b z^2}{2N} \cosh (\beta_{c,N}) \frac{\expec[W_{\sss N}^2]}{\expec[W_{\sss N}]} 
	+\frac{zr}{N}\frac{\expec[W_{\sss N}]}{\sqrt{\expec[W_{\sss N}^2]}} + o(1/N)	
	\label{abcd}
	\end{align}
{In the expansion of $G_N(z/N^{1/(\delta+1)};r)$ the first term in \eqref{abcd} drops as usual, whereas the second term in \eqref{abcd} remains.}
After multiplication by $N$ (cf.~\eqref{mulp-N-GN-conv}), one has 
\be
- N G_N(z/N^{1/(\boldsymbol{\delta}+1)};r) = \frac{b z^2}{2} \cosh (\beta_{c}) \frac{\expec[W_{\sss N}^2]}{\expec[W_{\sss N}]} 
- f\left(\frac{\expec[W_{\sss N}]}{\sqrt{\expec[W_{\sss N}^2]}}z\right) + o(1)
\ee
Using the substitution $x=\frac{\expec[W_{\sss N}]}{\sqrt{\expec[W_{\sss N}^2]}}z$
the above converges in the limit $N\to\infty$ to the exponent in \eqref{lim-dens-scal-window}, as required.
}
\qed

\paragraph{Limit distribution at $\beta_c$ instead of $\beta_{c,N}$.}
In the above, we look at the inverse temperature sequence \co{$\beta_N= \beta_{c,N}$} and then take the limit $N\rightarrow\infty$. Alternatively, we could immediately start with $\beta= \beta_{c}$. 
\co{The scaling limit that will be seen depends on
the speed at which $\nu_{\sss N}$  approaches $\nu$. Indeed, from \eqref{betac_ann} and 
\eqref{ctn}, one has $\beta_c - \beta_{c,N} = O (\nu - \nu_{\sss N})$.}

We investigate this for the deterministic weights according to~\eqref{eq-precisepowerlaw}, and first investigate how close $\nu_{\sss N}$ is to $\nu$. By \cite[Lemma 2.2]{BhaHofLee09b}, $\nu_{\sss N}=\nu+ \zeta N^{-\eta} + o(N^{-\eta})$ with $\eta=(\tau-3)/(\tau-1)$ and $\zeta$ an explicit non-zero constant. Thus, for $\tau>5$, $\nu_{\sss N}=\nu+o(N^{-1/2})$. Hence, the results stay the same (see the previous discussion).

When $\tau\in(3,5)$, instead, $\nu_{\sss N}=\nu+ \zeta N^{-\eta} + o(N^{-\eta})=\nu+ \zeta N^{-\frac{\boldsymbol{\delta}-1}{\boldsymbol{\delta}+1}} + o(N^{-\frac{\boldsymbol{\delta}-1}{\boldsymbol{\delta}+1}})$, so we are shifted inside the critical window (see the previous discussion). Hence, in this case the limiting distribution changes.
\qed
\medskip

\begin{center}
\begin{table}
\caption{List of symbols used}
\medskip
\begin{tabular}{l|l|l}
\textbf{Symbol} & \textbf{Definition} & \textbf{Description} \\ \hline
$N$ & \ & Number of vertices \\
$[N]$ & $\{1,\ldots,N\}$ & Set of first $N$ positive integers \\
$\beta$ & & Inverse temperature \\
$B$ & & External field \\
$H_{\sss N} $ && Hamiltonian \\
$Z_{\sss N} $ & & Partition function \\
$\phi$ & $\lim_{N\to\infty} \frac{1}{N} \log Z_N$ & Pressure of inhomogeneous Curie-Weiss model \\
$w_i$ & & Weight of vertex $i$ \\
$\boldsymbol{w}$ & $(w_1,\ldots,w_N)$ & Sequence of weights \\
$\ell_{\sss N}$ & $\sum_{i=1}^N w_i$ & Total weight \\
$GRG_{\sss N} (\boldsymbol{w})$ && Generalized random graph with weights $\boldsymbol{w}$ and $N$ vertices\\
$p_{ij}$ & $\frac{w_i w_j}{\ell_N+w_i w_j}$ & Probability of an edge between vertices $i$ and $j$ in $GRG_N(\boldsymbol{w})$ \\
$W_{\sss N} $ && Weight of uniformly chosen vertex \\
$W$ && Random variable chosen such that $W_{\sss N} \stackrel{\cal D}{\longrightarrow} W$ \\
$\nu$ & $\expec[W^2]/\expec[W]$ & Size-biased weight\\
$\nu_{\sss N}$ & $\expec[W_{\sss N}^2]/\expec[W_{\sss N}]$ & Its finite volume analogue\\
$\tau$ && Power-law exponent \\
$Q_{\sss N}$ && Law of the random graphs \\
$P_{\sss N}$ && Annealed Ising measure \\
$\psi_{\sss N}$ &$\frac{1}{N}\log Q_{\sss N}(Z_{\sss N})$ & Annealed pressure \\
$S_{\sss N}$ & $\sum_{i=1}^N \sigma_i$ & Total spin \\
$M_{\sss N}$ & $P_{\sss N}(S_{\sss N}/N)$ & Annealed  magnetization \\
$\chi_{\sss N}$ & $\frac{\partial}{\partial B} M_{\sss N}$ & Annealed susceptibility \\
$\psi, M, \chi$ && $\lim_{N\to\infty}$ of $\psi_{\sss N}, M_{\sss N}, \chi_{\sss N}$, respectively\\
$\beta_c$ & $\asinh(1/\nu)$ & Annealed critical inverse temperature \\
$\beta_{c,{\sss N}}$ & $\asinh(1/\nu_{\sss N})$ & Its finite volume analogue \\
$\boldsymbol{\beta},\boldsymbol{\delta},\boldsymbol{\gamma},\boldsymbol{\gamma'}$ && Critical exponents, see Def.~\ref{def-CritExp} \\
$z^{*}$ && Fixed point of~\eqref{fixpGRG0} \\
$\widetilde{P}_{\sss N}$ && Curie-Weiss approximation of $P_{\sss N}$, see~\eqref{eq-ptildemeasure} \\
$\widetilde{Z}_{\sss N}$ && Curie-Weiss approximation of $Z_{\sss N}$, see~\eqref{eq-Ztildedef} \\
\hline
\end{tabular}
\end{table}
\end{center}

\vspace{0.4cm}
\noindent
{\small
{\bfseries Acknowledgments.}
\change{We thank Aernout van Enter for helpful discussions on inhomogeneous versions of the Curie-Weiss models.}
We thank Institute Henri Poincar\'e for the hospitality during the trimester ``Disordered systems, random spatial processes and their applications''. 
We acknowledge financial support from  the Italian Research Funding Agency (MIUR) through FIRB project 
grant n.\ RBFR10N90W.
The work of RvdH is supported in part by the Netherlands
Organisation for Scientific Research (NWO) through VICI grant 639.033.806 and the Gravitation {\sc Networks} grant 024.002.003.}

{\small
}

\end{document}